\documentclass[a4paper,reqno,oneside]{amsart}
\usepackage{hyperref}
\numberwithin{equation}{section}

\newtheorem{thm}{Theorem}
\newtheorem{lem}[thm]{Lemma}
\newtheorem{prop}[thm]{Proposition}
\theoremstyle{definition}
\newtheorem{exm}[thm]{Example}
\newtheorem{rem}[thm]{Remark}

\newcommand{\bbS}{\mathbb{S}}
\newcommand{\rmi}{\mathrm{i}}
\newcommand{\CC}{\mathbb{C}}
\newcommand{\RR}{\mathbb{R}}
\newcommand{\TTT}{\mathbb{T}}
\newcommand{\NN}{\mathbb{N}}
\newcommand{\ZZ}{\mathbb{Z}}
\newcommand{\cK}{\mathcal{K}}

\newcommand{\cN}{\mathcal{N}}

\newcommand{\cU}{\mathcal{U}}
\newcommand{\cV}{\mathcal{V}}
\newcommand{\cP}{\mathcal{P}}

\newcommand{\dd}{\mathrm{d}}

\DeclareMathOperator{\dist}{dist}
\DeclareMathOperator{\dom}{dom}

\title[]{Discrete spectrum of interactions\\ concentrated near conical surfaces}

\author{Thomas Ourmi\`eres-Bonafos}
\address{Laboratoire de Math\'ematiques d'Orsay, Univ.~Paris-Sud, CNRS, Universit\'e Paris-Saclay, 91405 Orsay, France}

\email{thomas.ourmieres-bonafos@math.u-psud.fr}
\urladdr{http://www.math.u-psud.fr/~ourmieres-bonafos/}

\author{Konstantin Pankrashkin} 

\address{Laboratoire de Math\'ematiques d'Orsay, Univ.~Paris-Sud, CNRS, Universit\'e Paris-Saclay, 91405 Orsay, France}

\email{konstantin.pankrashkin@math.u-psud.fr}
\urladdr{http://www.math.u-psud.fr/~pankrash/}

\begin{document}


\begin{abstract}
We study the spectrum of two kinds of operators involving a conical geometry:
the Dirichlet Laplacian in conical layers and Schr\"odinger operators with attractive $\delta$-interactions
supported by infinite cones.
Under the assumption that the cones have smooth cross-sections,
we prove that such operators have infinitely many eigenvalues accumulating below the threshold of the essential spectrum
and we express the accumulation rate in terms of the eigenvalues of an auxiliary
one-dimensional operator with a curvature-induced potential.
\end{abstract}                                                            

\maketitle

\section{Introduction}

\subsection{Problem setting and main results}
The study of Laplace-type operators in infinite regions attract a lot of attention
due to their importance in quantum physics. A particular attention is paid to geometrically
induced spectral properties with an important focus on the existence of eigenvalues. Such properties were studied for specific systems such as locally deformed tubes \cite{de95,ei01,es,GJ92} and more recently,
layers \cite{cek04,dek01}, for which it is known that suitably localized deformations
of straight tubes and layers may only produce finitely many eigenvalues.
The situation changes for ``long-range'' deformations even in very simple geometries:
as found in \cite{et}, the Dirichlet Laplacian in a circular conical layer has an infinite
discrete spectrum accumulating to the threshold of the essential spectrum.
The result was then improved in \cite{dobr} by calculating the precise accumulation
rate and similar effects were found for Schr\"odinger operators with $\delta$-interactions supported
by circular cones in \cite{bel,lob}. The previous papers used in an essential way the presence
of the rotational symmetry and the aim of the present work is to extend the study to conical layers
and conical surfaces with arbitrary smooth cross-sections. We will show that the associated
operators always have an infinite discrete spectrum and compute the accumulation rate of these eigenvalues
in terms of a one-dimensional operator acting on the cross-section.

Let us introduce the mathematical framework. By a \emph{conical surface} in $\RR^3$ we mean a Lipschitz hypersurface $S\subset \RR^3$
invariant under the dilations, i.e. $\lambda S=S$ for all $\lambda>0$. A conical surface $S$
is uniquely determined by its \emph{cross-section} $\gamma:=S\cap \bbS^2$, where $\bbS^2$
is the unit sphere in $\RR^3$ centered at the origin. If $\gamma$ is a $C^4$ smooth loop,
we say that $S$ has a smooth cross-section.

Pick a conical surface $S$ with a smooth cross-section $\gamma$ for the rest of the paper.
We are interested in the spectral properties
of two Laplace-type operators associated with $S$.
The first one, denoted $A_{S,d}$, $d>0$, is the Dirichlet Laplacian in the unbounded domain
\[
\Lambda_{S,d}:=\big\{ x\in\RR^3: \dist(x,S)<\tfrac{1}{2}\, d\big\}
\]
called the \emph{conical layer of width $d$} around $S$. The operator $A_{S,d}$ is rigorously defined
as the unique self-adjoint operator in $L^2(\Lambda_{S,d})$
generated by the quadratic form
\[F
a_{S,d}(u)=\iiint_{\Lambda_{S,d}} |\nabla u|^2\dd x, \quad u\in H^1_0(\Lambda_{S,d}),
\]
and it can be interpreted as a model of a quantum particle confined in a layer with a hard-wall boundary.
The second one, denoted $B_{S,\alpha}$, is the self-adjoint operator
in $L^2(\RR^3)$ generated by the quadratic form
\[
b_{S,\alpha}(u)=\iiint_{\RR^3}|\nabla u|^2 \dd x-\alpha \iint_S |u|^2\dd\sigma, \quad u\in H^1(\RR^3),
\]
where $\alpha>0$ is a constant and $\sigma$ is the two-dimensional Hausdorff measure on $S$.
Informally, the operator $B_{S,\alpha}$ acts as the distributional Laplacian on $\RR^3\setminus S$
on the functions $u$ satisfying $[\partial u]+\alpha u=0$ on $S$, where $[\partial u]$
is the jump of the normal derivative, and it can be interpreted as
a Schr\"odinger operator with an attractive $\delta$-potential of strength $\alpha$ keeping
a particle in a vicinity of the surface $S$, see e.g. \cite{AGHH}, \cite[Chapter 10]{ekbook}
and the review \cite{e08} for a detailed discussion.

One easily sees that, due to the invariance of $S$ with respect to the dilations,
the role of the parameters $d>0$ and $\alpha>0$ in the above definitions
is quite limited, as one has the unitary equivalences
$A_{S,d}\simeq d^{-2} A_{S,1}$ and $B_{S,\alpha}\simeq \alpha^2 B_{S,1}$. Hence, in what follows we set
\[
\Lambda_S:=\Lambda_{S,1},
\quad
a_S:=a_{S,1},
\quad
A_S:=A_{S,1},
\quad
b_S:=b_{S,1},
\quad
B_S:=B_{S,1}
\]
and study the normalized operators $A_S$ and $B_S$.

As already mentioned above, it seems that the case of conical geometries
was first considered in the paper~\cite{et} for the operator $A_S$.
For the particular case when $S$ is a circular cone it was shown that $A_S$ has infinitely many eigenvalues below the essential spectrum. The accumulation rate
of the eigenvalues was then computed in \cite{dobr}. As for the operator $B_S$,
it was first considered in \cite{bel}, in which it was shown that
if $S$ is a circular cone, then one has an infinite discrete spectrum.
The accumulation rate was then calculated in \cite{lob}.
The paper \cite{bp16} studied general conical surfaces and an
expression for the bottom of the essential spectrum of $B_S$ was obtained.
The paper \cite{el} contains first results on the discrete spectrum
of the operator $B_S$ for conical surfaces $S$ with arbitrary
smooth cross-sections, and the authors showed that there is at least
one eigenvalue below the essential spectrum. They also
posed an open question on whether or not
the discrete spectrum is always infinite. In the present paper,
in particular, we give an affirmative answer to this question.
Remark that the papers \cite{bpp,klob,p16} studied similar questions
for Robin Laplacians or Aharonov-Bohm operators on conical domains,
and the eigenvalue behavior appears to be quite different.

If the cross-section $\gamma$ is a great circle (i.e. a circle of maximal
radius $1$), then the surface $S$ is a plane and both
$A_S$ and $B_S$ admit a separation of variables: one has
$\sigma(A_S)=[\pi^2,+\infty)$ and $\sigma(B_S)=[-\tfrac 14 ,+\infty)$.
Therefore, in what follows we assume that
\begin{equation}
 \label{great}
\gamma \text{ is not a great circle (i.e. $S$ is not a plane).}
\end{equation}
Denote by $\ell>0$ the length of $\gamma$ and set
$\TTT = \RR/\ell\ZZ$. Furthermore, choose an arc-length parametrization of $\gamma$,
i.e. an injective $C^4$ function $\Gamma:\TTT\to \RR^3$ such that
$\Gamma(\TTT)=\gamma$ and $|\Gamma'|\equiv 1$ and set
\[
n:= \Gamma\times\Gamma'.
\]
Recall that the \emph{geodesic curvature}
$\kappa(s)$ of $\gamma$ at a point $\Gamma(s)$ is defined
through
\[
		n'(s) = \kappa(s)\Gamma'(s), \quad \text{i.e. } \kappa=(\Gamma\times \Gamma'')\cdot \Gamma',
\]
and the assumption \eqref{great} takes the form
\begin{equation}
 \label{great2}
\kappa\not\equiv 0.
\end{equation}
An important role will be played by the curvature-induced Schr\"odinger operator in $L^2(\TTT)$,
\begin{equation}
       \label{ksoper}
\cK_S=-\dfrac{\dd^2}{\dd s^2} -\dfrac{\kappa^2}{4}
\end{equation}
defined on $H^2(\TTT)$. The operator $\cK_S$ has compact resolvent, hence, its spectrum is a sequence
of eigenvalues $\lambda_j(\cK_S)$, $j\in\NN$, enumerated in the non-decreasing order and with multiplicities taken into account,
such that $\lim_{j\to+\infty} \lambda_j(\cK_S)=+\infty$.
In particular, the following quantity is well-defined:
\begin{equation}
      \label{eq-ks}
k_S:=\frac{1}{2\pi}\sum_{j\in\NN: \lambda_j(\cK_S)<0}\sqrt{-\lambda_j(\cK_S)}.
\end{equation}
The following assertion is almost obvious:
\begin{prop}\label{prop1}
Under Assumption \eqref{great} one has $k_S>0$.
\end{prop}
\begin{proof} It is sufficient to show that $\lambda_1(\cK_S)<0$.
By the min-max principle one has
\[
\lambda_1(\cK_S)\le \dfrac{\langle 1,\cK_S 1\rangle_{L^2(\TTT)}}{\langle 1,1\rangle_{L^2(\TTT)}}=-\dfrac{1}{4\ell} \displaystyle\int_\TTT \kappa^2\dd s,
\]
and the right-hand side is strictly negative due to \eqref{great2}.
\end{proof}

The main results of the paper are presented in the following two theorems.
For a self-adjoint operator $T$, let $\sigma(T)$ and $\sigma_\mathrm{ess}(T)$
denote its spectrum and essential spectrum, respectively.
If $T$ is semibounded from below and $E\le\inf\sigma_\mathrm{ess}(T)$, then
$\cN_E(T)$ denotes the number of eigenvalues of $T$ in $(-\infty,E)$,
and the map $E\mapsto \cN_E(T)$ is called the eigenvalue counting function
of $T$.

\begin{thm}[Dirichlet Laplacian in a conical layer]\label{thm1}
There holds 
\begin{gather}
  \label{asy1ess}
\sigma_\mathrm{ess}(A_S)=[\pi^2,+\infty),\\
   \label{asy1}
\cN_{\pi^2-E}(A_S)=k_S |\ln E| +o(\ln E) \text{ as }  E\to 0^+.
\end{gather}
In particular, the operator $A_S$ has infinitely many eigenvalues in $(-\infty,\pi^2)$.
\end{thm}

\begin{thm}[$\delta$-interaction on a conical surface]\label{thm2}
There holds 
\begin{gather}
  \label{asy2ess}
\sigma_\mathrm{ess}(B_S)=\big[-\tfrac{1}{4},+\infty\big),\\
   \label{asy2}
\cN_{-\frac{1}{4}-E}(B_S)=k_S |\ln E| +o(\ln E) \text{ as }  E\to 0^+.
\end{gather}
In particular, the operator $B_S$ has infinitely many eigenvalues in $\big(-\infty,-\frac 14\big)$.
\end{thm}

\begin{exm}\label{circle}
If $S$ is a circular cone of opening angle $2\theta$, $\theta\in (0,\frac\pi 2)$,
then the cross-section is a circle of geodesic radius $\theta$
having the length $\ell=2\pi \sin\theta$ and the constant geodesic
curvature $\kappa=\cot\theta$. One easily computes
\begin{gather*}
\lambda_1(\cK_S)=-\dfrac{\kappa^2}{4}=-\dfrac{\cot^2\theta}{4},
\quad
\lambda_2(\cK_S)=\dfrac{4\pi^2}{\ell^2}-\dfrac{\kappa^2}{4}=\dfrac{4-\cos^2\theta}{4\sin^2\theta}>0,\\
k_S=\dfrac{1}{2\pi} \sqrt{-\lambda_1(\cK_S)}= \dfrac{\cot\theta}{4\pi}. 
\end{gather*}
Therefore, for this particular case, the result of Theorem~\ref{thm1} coincides with Theorem 1.4 in~\cite{dobr},
while Theorem~\ref{thm2} is exactly Theorem 1.4 in \cite{lob}. 
\end{exm}

One can use the above computation to improve the result of Proposition~\ref{prop1} as follows:

\begin{thm}\label{thm5}
For a conical surface $S$ with a smooth cross-section of length $\ell\le 2\pi$
there holds
\[
k_S\ge \dfrac{\sqrt{4\pi ^2-\ell^2}}{4\pi \ell},
\]
and the equality holds iff $S$ is a circular cone.
\end{thm}

\begin{proof}
Let $\gamma$ be the cross-section of $S$ and $A$ be the area of the spherical domain enclosed
by $\gamma$ such that the vector $n=\Gamma\times\Gamma'$ points to its exterior.
The classical isoperimetric inequality for spherical domains, see \cite{rado}, reads as
$\ell^2\ge A (4\pi -A)$ or, equivalently, $(2\pi -A)^2\ge 4\pi^2-\ell^2$.
Due to the Gauss-Bonnet theorem one has
\[
\int_\TTT \kappa \,\dd s=2\pi - A,
\]
and using the Cauchy-Schwarz inequality we obtain
\[
4\pi^2-\ell^2\le (2\pi - A)^2=\Big( \int_\TTT \kappa\,\dd s\Big)^2
\le \int_\TTT 1\,\dd s \cdot \int_\TTT \kappa^2 \dd s=\ell \int_\TTT \kappa^2 \dd s
\]
thus
\[
\int_\TTT \kappa^2\dd s\ge \dfrac{4\pi^2-\ell^2}{\ell}.
\]
As previously, by the min-max principle we have
\begin{equation}
  \label{eq-kkk2}
\lambda_1(\cK_S)\le 
\dfrac{\langle 1,\cK_S 1\rangle_{L^2(\TTT)}}{\langle 1,1\rangle_{L^2(\TTT)}}=-\dfrac{1}{4\ell} \int_\TTT \kappa^2 \dd s
\le -\dfrac{4\pi^2 - \ell^2}{4\ell^2}.
\end{equation}
As already seen in Example~\ref{circle}, for circular cones one has the equality in \eqref{eq-kkk2}.
Assume now that $S$ is not a circular cone, then $\kappa$ is non-constant and the test function $1$
is not an eigenfunction of $\cK_S$, hence, the inequality in \eqref{eq-kkk2} is strict,
and
\[
k_S\ge \dfrac{1}{2\pi}\sqrt{-\lambda_1(\cK_S)}> \dfrac{\sqrt{4\pi ^2-\ell^2}}{4\pi \ell}. \qedhere
\]
\end{proof}

\begin{rem}
The result of Theorem~\ref{thm5} can be viewed as a kind of isoperimetric inequality:
among the conical surfaces with smooth cross-sections of fixed length $l \le 2\pi$,
the circular cones give the highest rate for the accumulation of discrete eigenvalues
to the bottom of the essential spectrum for both $A_S$ and $B_S$.
Remark that the first eigenvalue of $B_S$ is also maximized by the circular cones, see~\cite{el}.
\end{rem}

\begin{rem}\label{rem6}
Contrary to the case of circular cones, the sum in the definition of $k_S$
can contain an arbitrary large number of summands.
Namely, let $\ell>0$ and $m\in\NN$. For small $\varepsilon>0$
one can construct a smooth loop $\gamma_\varepsilon$  on $\bbS^2$ having the length
$\ell$ and the following property: there exists $\Gamma:\TTT\to \RR^3$, an arc-length parametrization of $\gamma_\varepsilon$, such that
the finite pieces $\Gamma\Big(\big(\frac{j\ell}{m}-\varepsilon,\frac{j\ell}{m}+\varepsilon\big)\Big)$,
$j=1\dots,m$,
coincide with circular arcs of geodesic radius $\varepsilon$
and of length $2\varepsilon$. It follows that
$\kappa(s)=\cot\varepsilon$ for $s\in \big(\frac{j\ell}{m}-\varepsilon,\frac{j\ell}{m}+\varepsilon\big)$,
$j=1,\dots,m$. Furthermore, one can choose $m$ functions $\varphi_j\in C^\infty_c(\TTT)$, $j=1,\dots,m$, independent of $\varepsilon$,
such that
\[
\mathop{\mathrm{supp}} \varphi_j\subset \Big(\frac{j\ell }{m}-\frac{\ell}{2m},\frac{j\ell}{m}+\frac{\ell}{2m}\Big),
\quad
\varphi_j(s)=1 \text{ for }s\in \Big(\frac{j\ell}{m}-\varepsilon,\frac{j\ell}{m}+\varepsilon\Big),
\quad
\|\varphi_j\|_{L^2(\TTT)}=1,
\]
then for $u=\sum_{j=1}^m\alpha_j\varphi_j$, $\alpha=(\alpha_1,\dots,\alpha_m)\in\CC^{m}\setminus\big\{(0,\dots,0)\big\}$,
one has
\begin{multline*}
\dfrac{\langle u,\cK_S u\rangle_{L^2(\TTT)}}{\|u\|^2_{L^2(\TTT)}}
=\dfrac{\sum_{j=1}^m  |\alpha_j|^2\Big(\|\varphi'_j\|^2_{L^2(\TTT)}  - \tfrac{1}{4}\|\kappa \varphi_j\|^2_{L^2(\TTT)}\Big)}{\|\alpha\|_{\CC^m}^2}\\
\le\dfrac{\sum_{j=1}^m |\alpha_j|^2\|\varphi'_j\|^2_{L^2(\TTT)}}{\|\alpha\|_{\CC^m}^2}-\dfrac{\varepsilon\cot^2\varepsilon}{2}
\le C -\dfrac{\varepsilon\cot^2\varepsilon}{2}, \quad C:=\max\|\varphi'_j\|^2_{L^2(\TTT)}
\end{multline*}
By the min-max principle it follows that $\lambda_m(\cK_S)\le C -\varepsilon\cot^2\varepsilon/2<0$ as $\varepsilon$
is sufficiently small, and then there are at least $m$ summands in \eqref{eq-ks}. The example also shows that,
at a fixed length of the cross-section, there is no finite upper bound for $k_S$.
\end{rem}

\begin{rem}
It would be interesting to understand whether the results can be extended
to the case of a conical surface $S$ whose cross-section $\gamma$
is an \emph{open} smooth arc. As will be seen from the proofs,
a literal adaptation of our approach only gives a two-sided estimate,
\[
k^D_S |\ln E| +o(\ln E)\le \cN_E \le k^N_S |\ln E| +o(\ln E), \quad E\to 0^+,
\]
with $\cN_E$ standing for either $\cN_{\pi^2-E}(A_S)$
or $\cN_{-\frac14-E}(B_S)$ and
\[
k^{D/N}_S=\frac{1}{2\pi}\sum_{j\in\NN: \lambda_j(\cK^{D/N}_S)<0}\sqrt{-\lambda_j(\cK^{D/N}_S)},
\]
where $\cK^{D/N}_S$ is given by the same differential expression \eqref{ksoper}
but acts on the functions satisfying Dirichlet/Neumann boundary conditions
at the endpoints of $\gamma$. By analogy with the recent works
on Schr\"odinger operators with strong $\delta$-interactions \cite{dekp,ep}
we conjecture that the asymptotics of Theorems~\ref{thm1} and \ref{thm2}
still hold with $k_S^D$ instead of $k_S$.
\end{rem}

Both Theorem~\ref{thm1} and Theorem~\ref{thm2}
are proved by estimating the quadratic forms $a_S$
and $b_S$ using curvilinear coordinates in adapted
tubular neighborhoods of $S$. After suitable cut-offs,
we reduce the problem to the study of some
one-dimensional models for which the asymptotics
of the eigenvalue counting function is known (see Proposition~\ref{th:KS88} below).
The proof of Theorem~\ref{thm1} is given in Section~\ref{sec:spec_lay},
while the proof of Theorem~\ref{thm2}
is presented in Section \ref{sec:lb-delta} (the lower bound)
and Section~\ref{sec:ub-delta} (the upper bound and the essential spectrum).

\subsection{Preliminaries}
\label{sec:PI}

Let us list some conventions used throughout the text.
We denote $\NN=\{1,2,3,\dots\}$ and $\NN_0=\NN\mathop{\cup}\{0\}$.
For the geodesic curvature $\kappa$ defined on $\TTT$ one sets
\[
\kappa_\infty:=\|\kappa\|_\infty,
\quad
\kappa'_\infty:=\|\kappa'\|_\infty,
\quad
\kappa''_\infty:=\|\kappa''\|_\infty.
\]
Let $E\in\RR$. If $T$ is a self-adjoint operator, then we denote
by $\dom T$ its domain and by $\cN_E(T)$ the dimension
of the range of its spectral projector on $(-\infty,E)$.
If $T$ is lower semibounded and $E<\inf\sigma_\mathrm{ess}(T)$,
then $\cN_E(T)$ is exactly the number of eigenvalues of $T$ (counting the multiplicities)
in $(-\infty,E)$, otherwise one has $\cN_E(T)=+\infty$. Remark that
$\cN_E(T_1\oplus T_2)=\cN_E(T_1)+\cN_E(T_2)$ for any two self-adjoint operators
$T_1$ and $T_2$ and any $E\in\RR$.
By $\lambda_j(T)$, $j\in\NN$,
we denote the $j$-th eigenvalue of $T$ when enumerated in the non-decreasing
order and counted according to the multiplicities.
We recall that the function $E\mapsto \cN_E(T)$ is usually referred to
as the eigenvalue counting function for $T$.

If a self-adjoint operator $T$ is generated
by a closed lower semibounded
quadratic form $t$ defined on the domain $\dom t$,
then by definition $\cN_E(t):=\cN_E(T)$ and $\lambda_j(t):=\lambda_j(T)$, $j\in\NN$.
For two quadratic forms $t_1$ and $t_2$, their direct sum
$t_1\oplus t_2$ is the quadratic form $(t_1\oplus t_2)(u_1,u_2):=t_1(u_1)+t_2(u_2)$
defined for $(u_1,u_2)\in \dom(t_1\oplus t_2):=\dom t_1 \times \dom t_2$.
If $T_1$ and $T_2$ are the operators associated with $t_1$ and $t_2$,
then the operator associated with $t_1\oplus t_2$ is $T_1\oplus T_2$.
The form inequality $t_1\ge t_2$ means that $\dom t_1\subseteq \dom t_2$
and $t_1(u)\ge t_2(u)$ for all $u\in \dom t_1$. By the min-max principle,
the form inequality implies the reverse inequality
for the eigenvalue counting functions,
$\cN_E(t_1)\le \cN_E(t_2)$, for all $E\in\RR$.

For further references, let us recall the well-known Sobolev inequality,
see e.g. \cite[Lemma 8]{kuch},
\begin{equation}
	\big|u(0)\big|^2 \leq a \|u'\|_{L^2(0,b)}^2 + \frac2{a}\|u\|_{L^2(0,b)}^2 \text{ for }
	0<a\le b \text{ and } u\in H^1(0,b).
\label{eqn:Sob_ineq}
\end{equation}

For $L > 0$, consider the following two quadratic forms in $L^2(-L,L)$:
\[
	 q_{L,D/N} (u)  := \|u'\|^2_{L^2(-L,L)} - \big|u(0)\big|^2,\quad
	\dom q_{L,D} = H^1_0(-L,L), \quad  \dom q_{L,N} = H^1(-L,L),
\]
which are closed and semibounded from below. Hence, they generate self-adjoint operators
$Q_{L,D}$ and $Q_{L,N}$ in $L^2(-L,L)$. One easily checks that $Q_{L, D/N}$
acts as minus the second derivative in $(-L,0)\cup(0,L)$ on the functions $u$
satisfying $u(0^-)=u(0^+)=:u(0)$ and $u'(0^+)-u'(0^-)=-u(0)$ at the origin
and the Dirichlet/Neumann boundary conditions at the endpoints.
The understanding of the first two eigenvalues of $Q_{L,D/N}$ will be important for our purposes. 
The next proposition is proven in \cite[Propositions 2.4 and 2.5]{EY02}:
\begin{prop}\label{prop:mod1D}
There exist $L_0>0$ and $C_0>0$ such that for $L\ge L_0$ one has
	\begin{gather}
	  \label{1d1}
			-\tfrac{1}{4} -	C_0^{-1} e^{- C_0 L}
			\le
			\lambda_1(Q_{L,N})
			\le
			-\tfrac{1}{4}
			\le
			\lambda_1(Q_{L,D})
			\le
			-\tfrac{1}{4} +	C_0^{-1} e^{- C_0 L},\\
			\label{1d2}
			\lambda_2(Q_{L,D/N}) \ge 0.
	\end{gather}
\end{prop}

Further, we recall a result about another family of one-dimensional Schr\"{o}dinger operators,
which is a suitable reformulation of Theorem 1 in \cite{KS88}:
\begin{prop}\label{th:KS88}
Let $x_0\in\RR$ and $V:[x_0,+\infty)\to \RR$ be continuous with
$c:=\lim_{x\to +\infty} x^2 V(x) \in \RR$,
then the Schr\"odinger operator
$Q=-\dd^2/\dd x^2-V$ in $L^2(x_0,+\infty)$ with any self-adjoint boundary condition at $x=x_0$ satisfies
	\begin{equation*}
		\cN_{-E}(Q)= \frac{1}{2\pi}\, \sqrt{\Big(c - \frac14\Big)_+}\ |\ln E| + o(\ln E) \text{ as } E \to 0^+,
	\end{equation*}
	where $(x)_+:=\max(x,0)$ is the positive part of $x$.
\end{prop}

\section{Proof of Theorem~\ref{thm1}}
\label{sec:spec_lay}

\subsection{Cutting out the vertex}
\label{subsec:ftools}

It is a classical result of differential geometry that one can find $R_0>0$ such that
for all $R>R_0$ the map
\begin{equation}
 \label{prop:tub-layer}
\Phi:\Pi_R:=(R,+\infty)\times\TTT\times(-\tfrac 12,\tfrac 12)
\ni (r,s,t)\mapsto r\Gamma(s) + t n(s)\in \RR^3
\end{equation}
is injective, with $\dist\big(\Phi(r,s,t),S\big) = |t|$.
For $R>R_0$, we denote $\Omega_R:=\Phi(\Pi_R)$ and consider
the quadratic forms $a_{R,D/N} (u) = \| \nabla u \|_{L^2(\Omega_R)}^2$ defined
on
\[
\dom a_{R,D} = H_0^1(\Omega_R), \quad \dom a_{R,N} = \big\{ u \in H^1(\Omega_R): u = 0 \text{ on } \partial\Omega_R\cap \partial \Lambda_S\big\}.
\]
\begin{lem} \label{prop:D-N_bracket_lay}
 There exists $C_R>0$ such that for all $E>0$ there holds
	\[
		\cN_{\pi^2 - E}(a_{R,D}) \leq \cN_{\pi^2 - E}(A_S) \leq C_R + \cN_{\pi^2 - E}(a_{R,N}).
	\]
\end{lem}
\begin{proof}
Denote $U_R:=\Lambda_S\setminus \overline{\Omega}_R$. We have the form inequality $a'_R\oplus a_{R,N} \leq  a_S \leq a_{R,D}$,
where $a'_R(u) = \|\nabla u\|_{L^2(U_R)}^2$ with $\dom a'_R = \big\{ u \in H^1(U_R): u=0 \text{ on } \partial U_R\cap \partial\Lambda_S\big\}$.
Consequently, 
	\[
		\cN_{\pi^2 - E}(a_{R,D}) \leq \cN_{\pi^2 - E}(A_S) \leq \cN_{\pi^2 - E}(a'_R) + \cN_{\pi^2 - E}(a_{R,N}), \quad E>0.
	\]
	As $U_R$ is bounded Lipschtiz, the domain $\dom a_R'$ is compactly embedded into $L^2(U_R)$. Hence,
	the associated operator has compact resolvent, and the result holds for $C_R:=\cN_{\pi^2}(a'_R)<+\infty$.
\end{proof}

\subsection{Reformulation in tubular coordinates}
\label{subsec:tubcoor_fm}
Before going any further, we reformulate the problem using the tubular coordinates $(r,s,t)$
introduced in \eqref{prop:tub-layer}.

\begin{lem}
\label{prop:fatub_lay}
The quadratic forms $a_{R,D/N}$ are unitarily equivalent to the respective quadratic forms $b_{R,D/N}$ in $L^2\big(\Pi_R, (r+t\kappa)\dd r\,\dd s\,\dd t\big)$
defined by
	\begin{align*}
		b_{R,D/N}(u) & =  \iiint_{\Pi_R} \bigg((r + t\kappa)\big(|\partial_r u|^2 + |\partial_t u|^2\big) + \frac1{r + t \kappa}|\partial_s u|^2 \bigg)\dd r \,\dd s \,\dd t,\\
		\dom b_{R,N} & =  \Big\{u\in H^1_\mathrm{loc}(\Pi_R)\mathop{\cap} L^2\big(\Pi_R, (r+t\kappa)\dd r\,\dd s\,\dd t\big):\, b_{R,N}(u)<\infty, \\
		&\qquad u(\cdot,\pm\tfrac12)=0\text{ on }  (R,+\infty)\times \TTT\Big\},\\
		\dom b_{R,D} & =  \Big\{u \in \dom b_{R,N} : u(R,\cdot) = 0 \text{ on }  \TTT\times\big(-\tfrac12,\tfrac12\big)\Big\}.
	\end{align*}
\end{lem}

\begin{proof}
There holds $\partial_r\Phi = \Gamma$, $\partial_s\Phi = (r +t \kappa)\Gamma'$,
$\partial_t \Phi = n$,	and the associated metric tensor $G$ writes as 
	\[
		G = \big(\partial_p\Phi\cdot\partial_k\Phi\big)_{p,k\in\{r,s,t\}} =	\left(\begin{array}{ccc}
									1 & 0 & 0\\
									0 & (r +t \kappa)^2 & 0\\
									0 & 0 & 1
								\end{array}\right).
	\]
Set $g = \sqrt{\det G} = r+t\kappa$ and consider the unitary transform $U$,
\[
	U:L^2(\Omega_R) \to L^2(\Pi_R,g \,\dd r\,\dd s\,\dd t), \quad 	U v= v\circ\Phi.
\]
For $w\in\dom a_{R,D/N}$ we set $u(r,s,t) = (U w)(r,s,t) = w\big(\Phi(r,s,t)\big)$, then $u\in \dom b_{R,D/N}$.
Performing the change of variables, we get the quadratic forms on $L^2(\Pi_R,(r+t\kappa)\dd r\,\dd s \,\dd t)$:
	\[
		a_{R,D/N}(U w) = \iiint_{\Pi_R} \sum_{p,k\in\{r,s,t\}} (G^{-1})_{p,k}\partial_p u\partial_k u \,g \,\dd r \,\dd s \,\dd t
		=b_{R,D/N}(u),
	\]
	which gives the result.
\end{proof}

The next formulation of the problem allows to understand it on a $L^2$-space with the flat metric.

\begin{lem} 
	\label{prop:unit_2fmlay}
The quadratic forms $b_{R,D/N}$ from Lemma~\ref{prop:fatub_lay}
 are unitarily equivalent to the quadratic forms $c_{R,D/N}$ on $L^2(\Pi_R)$ defined as:
			\begin{align*}
						c_{R,N}(v) & = \iiint_{\Pi_{R}} \bigg(|\partial_r v|^2 + \frac1{(r+t\kappa)^2}\Big(|\partial_s v|^2 - \frac{\kappa^2 +1}{4}|v|^2\Big) + |\partial_t v|^2\\
			&\quad +\Big(\frac{t\kappa''}{2(r+t\kappa)^3} -\frac54\frac{(t\kappa')^2}{(r+t\kappa)^4}\Big)|v|^2\bigg) \dd r \,\dd s \,\dd t +  \int_{-\frac 12}^{\frac 12} \int_\TTT \frac{|v(R,s,t)|^2}{2(R+t\kappa)}\dd s\,\dd t,\\
			\dom c_{R,N}  &=  \Big\{v\in H^1_\mathrm{loc}(\Pi_R)\mathop{\cap} L^2(\Pi_R):\, c_{R,N}(v)<\infty,\  v(\cdot,\pm\tfrac12)=0\text{ on }  (R,+\infty)\times \TTT\Big\},\\
			c_{R,D}(v) & = \iiint_{\Pi_{R}} \bigg( |\partial_r v|^2 + \frac1{(r+t\kappa)^2}\Big(|\partial_s v|^2 - \frac{\kappa^2 +1}{4}|v|^2\Big)+ |\partial_t v|^2\\
			&\quad+\Big(\frac{t\kappa''}{2(r+t\kappa)^3} -\frac54\frac{(t\kappa')^2}{(r+t\kappa)^4}\Big)|v|^2 \bigg)\dd r \,\dd s \,\dd t,\\
			\dom c_{R,D} & =  \Big\{v \in \dom c_{R,N} : v(R,\cdot) = 0 \text{ on }  \TTT\times\big(-\tfrac12,\tfrac12\big)\Big\}.
		\end{align*}
\end{lem}

\begin{proof} Consider the unitary transform $V$,
\[
	V: L^2\big(\Pi_R,(r+t\kappa)\dd r\,\dd s\,\dd t\big) \to L^2(\Pi_R),
	\quad (V u)(r,s,t)= u(r,s,t)\sqrt{r+t\kappa(s)}.
\]
Let $j\in\{D,N\}$,  $u\in\dom b_{R,j}$ and $v := V u$, then $u = (r+t\kappa)^{-\frac 12}v$. By definition, one has $v\in\dom c_{R,j}$ and
\begin{align*}
		|\partial_r u|^2 & = \frac1{r+ t\kappa} |\partial_r v|^2 + \frac1{4(r+t\kappa)^3}|v|^2 - \frac1{2(r+t\kappa)^2}\partial_r\big(|v|^2\big),\\
		|\partial_s u|^2 & = \frac1{r+t\kappa}|\partial_s v|^2 + \frac{(t\kappa')^2}{4(r+t\kappa)^3}|v|^2 - \frac{t \kappa'}{2(r + t\kappa)^2}\partial_s\big(|v|^2\big),\\
		|\partial_t u|^2 & = \frac1{r+t\kappa}|\partial_t v|^2 + \frac{\kappa^2}{4(r+t\kappa)^3}|v|^2 - \frac{\kappa}{2(r+t\kappa)^2}\partial_t\big(|v|^2\big).
	\end{align*}
An integration by parts gives
\begin{align*}
	-\int_{R}^{+\infty}\frac1{2(r+t\kappa)}\partial_r\big(|v|^2\big)\dd r & =  \displaystyle\frac{|v(R,s,t)|^2}{2(R+t\kappa)} - \int_{R}^{+\infty}\frac1{2(r+t\kappa)^2}|v|^2\dd r,\\
		-\int_{\TTT}\frac{t\kappa'}{2(r+t\kappa)^3}\partial_s\big(|v|^2\big) \dd s & =   \int_{\TTT} \Big(\frac{t\kappa''}{2(r+t\kappa)^3} - \frac32\frac{(t\kappa')^2}{(r+t\kappa)^4}\Big)|v|^2\dd s,\\
		 -\int_{-\frac 12}^{\frac12} \frac{\kappa}{2(r+t\kappa)}\partial_t\big(|v|^2\big)\dd t & = -\int_{-\frac 12}^{\frac12}\frac{\kappa^2}{2(r+t\kappa)^2}|v|^2\,\dd t,
\end{align*}
and the substitution into  the expression of $b_{R,j}$ gives the sought equality $b_{R,j}(u) = c_{R,j}(V u)$.
\end{proof}

\subsection{Bounds for the quadratic forms}
\label{subsec:bound-fq}
Thanks to Lemmas \ref{prop:D-N_bracket_lay}, \ref{prop:fatub_lay} and \ref{prop:unit_2fmlay},
in order to prove Theorem \ref{thm1} it is sufficient to find suitable bounds
for the quadratic forms $c_{R,D/N}$.
\begin{lem} There exist two constants $B_{R,D/N}>0$ such that $c_{R,D} \leq f_{R,D}$
and $f_{R,N} \leq c_{R,N}$, where the quadratic forms $f_{R,D/N}$ are defined on $\dom f_{R,D/N}=\dom c_{R,D/N}$ by
\begin{align*}
	f_{R,D}(v) &=\iiint_{\Pi_R} \bigg(|\partial_t v|^2 + |\partial_r v|^2 + \frac{1}{(r-\frac{\kappa_\infty}2)^2}\Big(|\partial_s v|^2 - \frac{\kappa^2 + 1}{4}|v|^2\Big)+ \frac{B_{R,D}}{(r-\frac{\kappa_\infty}{2})^3}|v|^2\bigg)\dd r\,\dd s\, \dd t,\\
	f_{R,N}(v) &=\iiint_{\Pi_R} \bigg(|\partial_t v|^2 + |\partial_r v|^2 + \frac1{(r+\frac{\kappa_\infty}2)^2}\Big(|\partial_s v|^2 - \frac{\kappa^2 + 1}{4}|v|^2\Big)- \frac{B_{R,N}}{(r+\frac{\kappa_\infty}{2})^3}|v|^2 \bigg)\dd r\,\dd s\,\dd t.
\end{align*}
\label{prop:bounds_lay}
\end{lem}

\begin{proof} For $v\in\dom c_{R,D}$ we have
\begin{multline}
	c_{R,D}(v) \leq \iiint_{\Pi_R}\bigg(|\partial_t v|^2 + |\partial_r v|^2 +  \frac1{(r-\frac{\kappa_\infty}2)^2}|\partial_s v|^2\\
	- \frac1{(r+\frac{\kappa_\infty}2)^2}\frac{\kappa^2 + 1}4|v|^2+ \frac{\kappa_\infty''}{4(r - \frac{\kappa_\infty}{2})^3}|v|^2\bigg)\dd r\,\dd s \,\dd t.
\label{eqn:bound1lay}
\end{multline}
An easy computation yields
\begin{equation}
	\frac1{(r-\frac{\kappa_\infty}2)^2} - \frac1{(r+\frac{\kappa_\infty}2)^2} = \frac{2r\kappa_\infty}{(r-\frac{\kappa_\infty}2)^2(r+\frac{\kappa_\infty}2)^2},
\label{eqn:easy_comp}
\end{equation}
which, combined with \eqref{eqn:bound1lay}, gives
\begin{align*}
c_{R,D}(v) &\leq  \iiint_{\Pi_R}\bigg(|\partial_t v|^2 + |\partial_r v|^2 + \frac1{(r-\frac{\kappa_\infty}{2})^2}\Big(|\partial_s v|^2 - \frac{\kappa^2 +1}{4}|v|^2	\Big) +w_D |v|^2\bigg) 	\dd r \,\dd s \,\dd t,\\
w_D(r)&=\frac{\kappa_\infty^2 + 1}{4} \frac{2r\kappa_\infty}{(r-\frac{\kappa_\infty}2)^2(r+\frac{\kappa_\infty}2)^2} + \frac{\kappa_\infty''}{4(r-\frac{\kappa_\infty}{2})^3}.
\end{align*}
The function $r \mapsto (r-\frac 12 \kappa_\infty)^3 w_D(r)$ is continuous on $[R,+\infty)$ and has a finite limit as $r\rightarrow+\infty$.
Hence, there exists $B_{R,D}>0$ such $w_D(r)\leq B_{R,D} (r-\frac 12 \kappa_\infty)^{-3}$ for all $r\in[R,+\infty)$, which gives $c_{R,D} \leq f_{R,D}$.

Now let $v\in\dom c_{R,N}$, then
\begin{multline*}
	c_{R,N}(v) \geq  \iiint_{\Pi_R} \bigg(|\partial_t v|^2 + |\partial_r v|^2 + \frac{1}{(r+\frac{\kappa_\infty}{2})^2}|\partial_s v|^2 - \frac{1}{(r-\frac{\kappa_\infty}{2})^2}\frac{\kappa^2 + 1}{4}|v|^2\\
	- \Big(\frac{\kappa_\infty''}{4(r-\frac{\kappa_\infty}{2})^3} + \frac5{16}\frac{(\kappa_\infty')^2}{(r-\frac{\kappa_\infty}{2})^4}\Big)|v|^2\bigg)\dd r\,\dd s\,\dd t.
\end{multline*}
Taking into account \eqref{eqn:easy_comp}, it rewrites
\begin{align*}
	c_{R,N}(v) &\geq \iiint_{\Pi_R} \bigg(|\partial_t v|^2 + |\partial_r v|^2 + \frac{1}{(r+\frac{\kappa_\infty}{2})^2}\big(|\partial_s v|^2 - \frac{\kappa^2 + 1}{4}|v|^2\big)-w_N |v|^2\bigg) 	\dd r\,\dd s\,\dd t,\\
 w_N(r)&:=\frac{2r\kappa_\infty(\kappa_\infty^2 + 1)}{4(r - \frac{\kappa_\infty}{2})^2(r + \frac{\kappa_\infty}{2})^2} + \frac{\kappa_\infty''}{4(r-\frac{\kappa_\infty}{2})^3} + \frac5{16}\frac{(\kappa_\infty')^2}{(r-\frac{\kappa_\infty}{2})^4}.
\end{align*}
As the function $r \mapsto (r+\frac12 \kappa_\infty)^3 w_N(r)$ is continuous on $[R,+\infty)$ and has a finite limit as $r\rightarrow+\infty$,
there exists $B_{R,N}>0$ such that $w_N(r)\leq B_{R,N} (r+\frac12 \kappa_\infty)^{-3}$ for all $r\in[R,+\infty)$. This concludes the proof.
\end{proof}

Set $\rho_D := R - \frac{\kappa_\infty}{2}$, $\rho_N:= R + \frac{\kappa_\infty}{2}$,
and $\Pi_{R,D/N}:= (\rho_{D/N},+\infty)\times\TTT\times(-\tfrac12,\tfrac 12)$,
and consider the following quadratic forms in $L^2(\Pi_{R,D/N})$:
\begin{align*}
	g_{R,N}(v) &=\iiint_{\Pi_{R,N}}\bigg(|\partial_t v|^2 + |\partial_\rho v|^2 + \frac1{\rho^2}\Big(|\partial_s v|^2 - \frac{\kappa^2 + 1}{4}|v|^2\Big) - \frac{B_{R,N}}{\rho^3}|v|^2 \bigg)\dd\rho\,\dd s\,\,\dd t,\\
	\dom g_{R,N}&=\Big\{ v\in L^2(\Pi_{R,N}):\  \partial_\rho v,\rho^{-1}\partial_s v,\partial_t v\in L^2(\Pi_{R,N}), \  v(\cdot,\pm\tfrac12)=0\text{ on }  (\rho_N,+\infty)\times \TTT\Big\},\\
		g_{R,D}(v) &=\iiint_{\Pi_{R,D}}\bigg(|\partial_t v|^2 + |\partial_\rho v|^2 + \frac1{\rho^2}\Big(|\partial_s v|^2 - \frac{\kappa^2 + 1}{4}|v|^2\Big) + \frac{B_{R,D}}{\rho^3}|v|^2 \bigg)\dd\rho\,\dd s\,\dd t,\\
			\dom g_{R,D}&=\Big\{ v\in L^2(\Pi_{R,D}):\  \partial_\rho v,\rho^{-1}\partial_s v,\partial_t v\in L^2(\Pi_{R,D}), \  v(\cdot,\pm\tfrac12)=0\text{ on }  (\rho_D,+\infty)\times \TTT,\\
			&\qquad\qquad v(\rho_D,\cdot) = 0 \text{ on }  \TTT\times\big(-\tfrac12,\tfrac12\big)\Big\}.
\end{align*}
The quadratic forms $g_{R,D/N}$ are unitarily equivalent to $f_{R,D/N}$ as they simply correspond to the change of variables $\rho=r\mp \frac{\kappa_\infty}{2}$,
and the preceding constructions can be summarized as follows:
\begin{lem}\label{cor:up_loblay}
For any $R>R_0$ there exists $C_R>0$ such that
\[
	\cN_{\pi^2 - E}(g_{R,D}) \leq \cN_{\pi^2 - E}(A_S) \leq C_R + \cN_{\pi^2 - E}(g_{R,N}) \text{ for all } E>0.
\]
\end{lem}

\subsection{Families of one-dimensional operators}\label{subsec:fam_1D}
We remark that the operators $G_{R,D/N}$ associated with the forms $g_{R,D/N}$ admit a separation of variables.
Indeed, one has the representations
\[
	L^2(\Pi_{D/N}) \simeq L^2(\rho_{D/N},+\infty)\otimes L^2\big(\TTT\times(-\tfrac12,\tfrac12)\big),
	\quad
	L^2\big(\TTT\times(-\tfrac12,\tfrac12)\big)\simeq L^2(\TTT)\otimes L^2(-\tfrac12,\tfrac12),
\]
and the operator $G_{R,D/N}$ commutes with the operators $1\otimes(\cK_S\otimes 1)$ and $1\otimes (1\otimes P)$
with $P$ being the Dirichlet Laplacian in $L^2(-\tfrac12,\tfrac12)$.
Both $\cK_S$ and $P$ have discrete spectra,
and their eigenvalues are $\lambda_m(\cK_S)$ and $\lambda_n(P) = \pi^2n^2$, $m,n\in\NN$.
The decomposition with respect to the associated orthonormal basis of eigenfunctions shows that
the operator $G_{R,D/N}$ is unitarily equivalent to the direct sum
\[
	G_{R,D/N} \simeq \bigoplus_{m,n\in\NN} \big(G_{R,D/N}^{[m]} + \pi^2n^2\big),
\]
where $G_{R,D/N}^{[m]}$ is the one-dimensional operator acting on the Hilbert space $L^2(\rho_{D/N},+\infty)$ and
generated by the quadratic form $g_{R,D/N}^{[m]}$ defined as
\[
	g_{R,D/N}^{[m]}(v) := \int_{\rho_{D/N}} \bigg(|v'|^2 + \Big(\frac{\lambda_m(\cK_S) -\frac 14}{\rho^2} + \frac{C_{R,D/N}}{\rho^3}\Big)|v|^2 \bigg)\dd\rho,
\]
where $C_{R,D} = B_{R,D}$ and $C_{R,N} = - B_{R,N}$, on the domains
\[
	\dom g_{R,D}^{[m]} = H_0^1(\rho_D,+\infty),\quad\dom g_{R,N}^{[m]} = H^1(\rho_N,+\infty).
\]
As the constant $R$ can be chosen arbitrarily large, we assume from now on that
\[
	\frac{\lambda_1(\cK_S) -\frac 14}{\rho^2} + \frac{C_{R,D/N}}{\rho^3} \geq -3\pi^2 \text{ for all } \rho>\rho_{D/N}.
\]
Hence, for all $m\in\NN$ and $n\geq 2$ we have $G_{R,D/N}^{[m]} + \pi^2n^2 \geq \pi^2$,
and for any $E>0$ there holds
\[
	\cN_{\pi^2-E}(G_{R,D/N}) = \sum_{m,n\in\NN}\cN_{\pi^2-E}(G_{R,D/N}^{[m]} + \pi^2n^2) = \sum_{m\in\NN}\cN_{- E} (G_{R,D/N}^{[m]}).
\]
As $\lambda_m(\cK_S)$ tends to $+\infty$ as $m$ goes to $+\infty$, one can find $M\in\NN$ such that
\[
	\lambda_m(\cK_S)\geq 0 \text{ and }\frac{\lambda_m(\cK_S) -\frac14}{\rho^2} + \frac{B_{R,D/N}}{\rho^3} \geq 0 \text{ for all } \rho>\rho_{D/N}
	\text{ and } m\ge M+1.
\]
It follows that $G_{R,D/N}^{[m]}\geq0$ for $m\geq M+1$, therefore
\[
	\cN_{\pi^2-E}(G_{R,D/N}) = \sum_{m=1}^{M}\cN_{- E} (G_{R,D/N}^{[m]}).
\]
The asymptotics of each summand on the right-hand side is described by Proposition~\ref{th:KS88},
\[
\cN_{- E} (G_{R,D/N}^{[m]})=\dfrac{1}{2\pi} \sqrt{\big(-\lambda_m(\cK_S)\big)_+} |\ln E| +o(\ln E),
\quad E\to 0^+,
\]
and we arrive at \eqref{asy1} using Lemma \ref{cor:up_loblay}.

\subsection{Essential spectrum}\label{ess1}
It remains to show Eq.~\eqref{asy1ess} for the essential spectrum. Remark first that the asymptotics \eqref{asy1} shows
already that $\inf \sigma_\mathrm{ess}(A_S)=\pi^2$ thus, it is sufficient
to show that $[\pi^2,+\infty)\subset \sigma(A_S)$.
Remark that, by the above changes of variables, for a smooth function $\varphi\in \dom A_S$
vanishing in $\Lambda_S\setminus \Phi(\Pi_R)$ one has $A_S\varphi=0$ in $\Lambda_S\setminus \Phi(\Pi_R)$
and
\begin{multline}
   \label{vua}
VU A_S\varphi= \bigg[-\dfrac{\partial^2}{\partial r^2}-\dfrac{\partial}{\partial s} \Big( \dfrac{1}{(r+t\kappa)^2}\dfrac{\partial}{\partial s}\Big)
-\dfrac{\partial^2}{\partial t^2}\\
+\Big(\frac{t\kappa''}{2(r+t\kappa)^3} -\frac54\frac{(t\kappa')^2}{(r+t\kappa)^4} -\dfrac{\kappa^2+1}{4(r+t\kappa)^2}\Big)\bigg]VU\varphi
\text{ in } \Pi_R.
\end{multline}

Choose a $C^\infty$ function $\chi:\RR\to \RR$ with $\chi=0$ on $(-\infty, 0)$ and $\chi=1$ on $(1,+\infty)$
and  let $k\ge 0$. For $N>R$, define $\varphi_N\in \dom A_S$ by
\[
\varphi_N=0 \text{ on } \Lambda_S\setminus \Phi(\Pi_N),
\quad
(V U \varphi_N)(r,s,t)=e^{\rmi k r} \cos (\pi t)\chi(N-r)\chi(2N-r) \text{ for } (r,s,t)\in\Pi_R,
\]
then using \eqref{vua} one easily shows that
\[
\lim_{N\to+\infty} \dfrac{\big\|\big(A_S-(\pi^2+k^2)\big)\varphi_N\big\|_{L^2(\Lambda_S)}}{\|\varphi_N\|_{L^2(\Lambda_S)}}
=\lim_{N\to+\infty} \dfrac{\big\|VU A_S\varphi_N-(\pi^2+k^2)VU\varphi_N\big\|_{L^2(\Pi_R)}}{\|VU\varphi_N\|_{L^2(\Pi_R)}}=0,
\]
which means $\pi^2+k^2\in \sigma(A_S)$. As $k\ge 0$ is arbitrary, the result follows.

\section{Proof of Theorem \ref{thm2}: Lower bound}\label{sec:lb-delta}

The aim of this section is to obtain the lower bound
\begin{equation}
  \label{prop:born_inf_delta}
	\liminf_{E \to 0^+}\frac{\cN_{-\frac 14-E}(B_S)}{|\ln E|} \geq k_S.
\end{equation}

\subsection{Change of variables}
\label{subsec:first-tools-lb}

The construction of the tubular coordinates will be done in a slightly different form,
in order to allow a greater freedom in the choice of parameters. Let $R_0>0$,
then one can find $\delta_0\in(0,\kappa_\infty^{-1})$
such that for all $\delta\in(0,\delta_0)$ and $R>R_0$, the map
	\begin{equation}
	   \label{prop:tub_ne}
		\Lambda: \cV_{R,\delta}:=(R,+\infty)\times\TTT\times(-\delta R,\delta R)
		\ni (r,s,t)\mapsto r\Gamma(s) + t n(s)
	\end{equation}
is injective. From now on, we pick $R_0$ and $\delta_0$ satisfying the above conditions. For $R>R_0$ and $\delta\in(0,\delta_0)$ we denote $\Omega_{R,\delta}:=\Lambda(\cV_{R,\delta})$, $S_{R,\delta} := S \cap \Omega_{R,\delta}$
and consider the associated quadratic form in $L^2(\Omega_{R,\delta})$,
\[
	b_{R,\delta}(u) = \iiint_{\Omega_{R,\delta}} |\nabla u|^2\dd x - \iint_{S_{R,\delta}} |u|^2\dd\sigma, \quad \dom b_{R,\delta} = H_0^1(\Omega_{R,\delta}).
\]
As $b_{R,\delta}$ can be viewed as a restriction of $b_S$, we have $\cN_{-\frac 14-E}(b_{R,\delta}) \leq \cN_{-\frac 14-E}(B_S)$ for $E>0$.
Now we are concerned with a lower bound for the eigenvalue counting function for $b_{R,\delta}$
under a suitable choice of $R$ and $\delta$.

Proceeding as in Lemma~\ref{prop:fatub_lay} and in Lemma~\ref{prop:unit_2fmlay} we rewrite
the problem using the tubular coordinates $(r,s,t)$ introduced in \eqref{prop:tub_ne}
and see that the quadratic form $b_{R,\delta}$ is unitarily equivalent to the quadratic form $c_{R,\delta}$ on $L^2(\cV_{R,\delta})$ defined as
\begin{align*}
		c_{R,\delta}(v) &= \iiint_{\cV_{R,\delta}} \bigg( |\partial_r v|^2 +\frac1{(r+t\kappa)^2}\Big(|\partial_s v|^2 - \frac{\kappa^2 + 1}{4}|v|^2\Big) + |\partial_t v|^2\\
		&\quad+ \Big(\frac{t\kappa''}{2(r+t\kappa)^3} - \frac54\frac{(t\kappa')^2}{(r+t\kappa)^4}\Big)|v|^2\bigg)\dd r\,\dd s \,\dd t
						 -\int_{\TTT}\int_R^\infty |v(r,s,0)|^2\dd r\,\dd s,\\
		\dom c_{R,\delta} &=\big\{v\in L^2(\cV_{R,\delta}): \,\partial_r v,\,r^{-1}\partial_s v,\,\partial_t v\in L^2(\cV_{R,\delta}), \quad v=0 \text{ on } \partial\cV_{R,\delta}		\big\}.
\end{align*}
For $v\in\dom c_{R,\delta}$ we have:
\begin{multline*}
\frac1{(r+t\kappa)^2}(|\partial_s v|^2 - \frac{\kappa^2 + 1}{4}|v|^2) \leq \frac{1}{(r-\delta R\kappa_\infty)^2}|\partial_s v|^2 - \frac{1}{(r+\delta R\kappa_\infty)^2}\frac{\kappa^2+1}{4}|v|^2\\
= \frac{1}{(r-\delta R\kappa_\infty)^2}\big(|\partial_s v|^2 -\frac{\kappa^2+1}{4}|v|^2\big) + \Big(\frac{1}{(r-\delta R\kappa_\infty)^2} - \frac{1}{(r+\delta R\kappa_\infty)^2}\Big)\frac{\kappa^2+1}{4}|v|^2.
\end{multline*}
Remark that
\begin{multline*}
	\frac{1}{(r-\delta R\kappa_\infty)^2} - \frac{1}{(r+\delta R\kappa_\infty)^2} = \frac{4\delta rR\kappa_\infty}{(r-\delta R\kappa_\infty)^2(r+\delta R\kappa_\infty)^2}\\ \leq \frac{4\delta R\kappa_\infty}{(r-\delta R\kappa_\infty)^2(r+\delta R\kappa_\infty)} \leq \frac{4\delta R\kappa_\infty}{(r-\delta R\kappa_\infty)^3}\leq \frac{4R\kappa_\infty}{R-\delta R\kappa_\infty} \frac{\delta}{(r-\delta R\kappa_\infty)^2}\\
	\leq \frac{4\kappa_\infty}{1-\delta_0\kappa_\infty} \frac{\delta}{(r-\delta R\kappa_\infty)^2}
\end{multline*}
and
\begin{multline*}
	\frac{t\kappa''}{2(r+t\kappa)^3} - \frac54\frac{(t\kappa')^2}{(r+t\kappa)^4} \leq \frac{t\kappa''}{2(r+t\kappa)^3} \leq \frac{\delta R \kappa_\infty''}{2(r-R\delta\kappa_\infty)^3}
	\\
	\leq \frac{R \kappa_\infty''}{2R(1-\delta\kappa_\infty)} \frac{\delta}{(r-R\delta\kappa_\infty)^2}
\leq \frac{\kappa_\infty''}{2(1-\delta_0\kappa_\infty)}\frac{\delta}{(r-R\delta\kappa_\infty)^2}.
\end{multline*}
Therefore, there exists a constant $C>0$, independent of $(R,\delta)$,
such that $c_{R,\delta}  \leq  f_{R,\delta}$ with
\begin{multline*}
	f_{R,\delta}(v) := \iiint_{\cV_{R,\delta}} |\partial_t v|^2 + |\partial_r v|^2 + \frac{1}{(r-\delta R\kappa_\infty)^2}\Big(|\partial_s v|^2 - \frac{\kappa^2+1 - C\delta}{4}|v|^2\Big)\dd r\,\dd s\,\dd t\\ - \int_{r>R}\int_{s\in\TTT}|v(r,s,0)|^2\dd r\,\dd s, \quad
	 \dom f_{R,\delta}=\dom c_{R,\delta}.
\end{multline*}
Using the change of variable
\[
\rho = \dfrac{r - \delta R\kappa_\infty}{R(1-\delta\kappa_\infty)}
\]
one sees that
the quadratic form $f_{R,\delta}$ is unitarily equivalent to the quadratic form $g_{R,\delta}$,
\begin{multline*}
	g_{R,\delta}(v) := \iiint_{\cU_{R,\delta}} \bigg(|\partial_t v|^2 + \frac{1}{R^2(1-\delta\kappa_\infty)^2}|\partial_\rho v|^2\\
	+ \frac{1}{R^2(1-\delta\kappa_\infty)^2\rho^2}\Big(|\partial_s v|^2 - \frac{\kappa^2+1 -C\delta}{4}|v|^2\Big)
	\bigg)	\dd r\,\dd s\,\dd t - \int_{\TTT}\int_1^{+\infty}|v(\rho,s,0)|^2\dd\rho\,\dd s,
\end{multline*}
with $\cU_{R,\delta} := \big(1,+\infty\big)\times\TTT\times(-\delta R,\delta R)$ and
\[
\dom g_{R,\delta} =\Big\{v\in L^2(\cU_{R,\delta}): \,\partial_\rho v,\,\rho^{-1}\partial_s v,\,\partial_t v\in L^2(\cU_{R,\delta}), \quad v=0 \text{ on } \partial\cU_{R,\delta}		\Big\}.
\]
By construction we have
\begin{equation}
   \label{count1}
 \cN_{-\frac14-E}(B_S)\ge \cN_{-\frac14-E}(g_{R,\delta}), \quad E>0.
\end{equation}

\subsection{Family of one-dimensional operators}\label{subsec:fam-2D-op}

We remark that the operators $G_{R,\delta}$ associated with the form $g_{R,\delta}$ admit a separation of variables.
First one uses the identification
\[
	L^2(\cU_{R,\delta}) \simeq L^2(1,+\infty)\otimes L^2\big(\TTT\times(-\delta R,\delta R)\big),
	\quad L^2\big(\TTT\times(-\delta R,\delta R)\big)\simeq L^2(\TTT)\otimes L^2(-\delta R,\delta R),
	\]
and then remarks that $G_{R,\delta}$ commutes with the operators $1\otimes(\cK_S\otimes 1)$ and $1\otimes (1\otimes Q_{R\delta,D})$,
with $Q_{R\delta,D}$ defined in Subsection~\ref{sec:PI}. Both $\cK_S$ and $Q_{R\delta,D}$
have discrete spectra, and the operator $G_{R,\delta}$ is unitarily equivalent to the direct sum
\[
	G_{R,\delta} \simeq \bigoplus_{m,n\in\NN} \bigg(\dfrac{1}{R^2(1-\delta\kappa_\infty)^2}\,G_{R,\delta}^{[m]} + \lambda_n(Q_{R\delta,D})\bigg),
\]
where $G_{R,\delta}^{[m]}$ is the one-dimensional operator acting on the Hilbert space $L^2(1,+\infty)$ associated with the following
quadratic form $g_{R,\delta}^{[m]}$:
\[
	g_{R,\delta}^{[m]}[v] := \int_{1}^{+\infty} \Big(|v'|^2 + \frac{\lambda_m(\cK_S) -\dfrac{1-C\delta}{4}}{\rho^2}|v|^2\Big)\dd\rho,
	\quad \dom g_{R,\delta}^{[m]} = H_0^1(1,+\infty).
\]
Hence, one can estimate the eigenvalue counting function as follows:
\begin{multline}
\label{eqn:low_boundlay}
	\cN_{-\frac 1 4-E}(G_{R,\delta})=\sum_{m,n\in\NN} \cN_{-\frac 1 4-E}\bigg(\dfrac{1}{R^2(1-\delta\kappa_\infty)^2}\,G_{\delta}^{[m]} + \lambda_n(Q_{R\delta},D)\bigg)\\
		\geq \sum_{m=1}^{M_\delta} \cN_{-\frac 1 4-E} \bigg(\dfrac{1}{R^2(1-\delta\kappa_\infty)^2}\,G_{R,\delta}^{[m]} + \lambda_1(Q_{R\delta,D})\bigg) \\
		=\sum_{m=1}^{M_\delta} \cN_{-\frac14-E - \lambda_1(Q_{R\delta,D})}\bigg(\dfrac{1}{R^2(1-\delta\kappa_\infty)^2}\,G_{R,\delta}^{[m]}\bigg), \quad E>0,
\end{multline}
where
\[
M_\delta := \max \Big\{ m\in\NN: \lambda_{m}(\cK_S) - \dfrac{1-C\delta}{4}< 0\Big\}.
\]
The above constructions are valid for all $R>R_0$. Now, assuming that $E>0$ is sufficiently small,
we choose $R=R_\delta(E) := K_\delta |\ln E|$ with $K_\delta>0$ to be chosen later and set
\[
	\mu_\delta(E) := \Big(\tfrac14 + E + \lambda_1(Q_{R_\delta(E)\delta,D})\Big)R_\delta(E)^2(1-\delta\kappa_\infty)^2.
\]
Thanks to Proposition \ref{prop:mod1D}, for $E$ small enough we have:
\[
	\Big|\tfrac14 + \lambda_{1}(Q_{R_\delta(E)\delta,D})\Big| \leq C_0^{-1} e^{-C_0 K_\delta \delta |\ln E|} = C_0^{-1} E^{C_0 K_\delta \delta}.
\]
Consequently, for $K_\delta$ large enough one has
\[
	\mu_\delta(E) = (1-\delta\kappa_\infty)^2 K^2_\delta|\ln E|^2 E + o(E|\ln E|^2)\to 0^+ \text{ for } E\to 0^+.
\]
With the help of \eqref{eqn:low_boundlay} we get
\[
	\cN_{-\frac14-E}(G_{R_\delta(E),\delta}) \geq \sum_{m=1}^{M_\delta} \cN_{-\mu_\delta(E)}(G_{R_\delta(E),\delta}^{[m]}),
\]
and Proposition \ref{th:KS88} gives
\begin{multline*}
	\liminf_{E\rightarrow 0^+}\frac{\cN_{-1/4-E}(G_{R_\delta(E),\delta})}{|\ln E|}
	\geq
	\sum_{m=1}^{M_\delta} \lim_{E \rightarrow 0^+} \frac{\cN_{-\mu_\delta(E)}(G_{R_\delta(E),\delta}^{[m]})}{|\ln E|}\\
	=
	\sum_{m=1}^{M_\delta} \lim_{E \rightarrow 0^+} \frac{\cN_{-\mu_\delta(E)}(G_{R_\delta(E),\delta}^{[m]})}{\big|\ln \mu_\delta(E)\big|}\cdot \dfrac{\big|\ln\mu_\delta(E)\big|}{|\ln E|}
	=
	\sum_{m=1}^{M_\delta} \lim_{E \rightarrow 0^+} \frac{\cN_{-\mu_\delta(E)}(G_{R_\delta(E),\delta}^{[m]})}{\big|\ln \mu_\delta(E)\big|}\\
	= \frac{1}{2\pi}\sum_{m=1}^{M_\delta} \sqrt{\Big(-\lambda_m(\cK_S) - \tfrac 14 C\delta\Big)_+}.
\end{multline*}
Due to \eqref{count1} we arrive at 
\[
\liminf_{E\rightarrow 0^+}\frac{\cN_{-\frac14-E}(B_S)}{|\ln E|}
\geq 
\frac{1}{2\pi}\sum_{m=1}^{M_\delta} \sqrt{\Big(-\lambda_m(\cK_S) -\tfrac14 C\delta\Big)_+}.
\]
As the inequality is true for any $\delta\in(0,\delta_0)$ and the right-hand side tends to $k_S$
as $\delta\to 0^+$, we arrive at \eqref{prop:born_inf_delta}.

\section{Proof of Theorem~\ref{thm2}: Upper bound}
\label{sec:ub-delta}

In this section we are going to show the inequality
	\begin{equation}
\label{prop:born_sup_delta}
		\limsup_{E\rightarrow 0^+}\frac{\cN_{-\frac 14 -E}(B_S)}{|\ln E|} \leq k_S.
	\end{equation}

\subsection{Change of variables}

Contrary to the preceding cases, we will work on a neighborhood of $S$
which suitable expands at infinity. Namely, for $R>0$ and $\delta>0$ we denote
\begin{gather*}
		\cP_{R,\delta}:= \big\{(r,t)\in\RR^2: \,r>R,\  t\in(-\delta r,\delta r)\big\}
		\equiv \big\{(r,t)\in\RR^2: r>r_{R,\delta}(t)\big\},
		\  r_{R,\delta}(t):= \max\Big(R,\dfrac{|t|}{\delta}\Big),\\
\cV_{R, \delta} := \cP_{R,\delta}\times\TTT,
\end{gather*}
then there exist $R_0>0$ and $\delta_0\in(0, \kappa_\infty^{-1})$
such that for all $\delta\in(0,\delta_0)$ and all $R>R_0$ the map
	\[
		\Lambda: \cV_{R,\delta}\to \RR^3, \quad \Lambda(r,s,t)=r\Gamma(s) + t n(s),
	\]
is injective. We set $\Omega_{R,\delta}:=\Lambda(\cV_{R,\delta})$ and define the following quadratic form in $L^2(\Omega_{R,\delta})$:
\[
a_{R,\delta}(u) := \iiint_{\Omega_{R,\delta}}|\nabla u|^2 \dd x - \iint_{S\mathop{\cap} \Omega_{R,\delta}}|u|^2\dd\sigma,
\quad
\dom a_{R,\delta} = H^1(\Omega_{R,\delta}).
\]

\begin{lem}\label{lem14}
Let $R>R_0$ and $\delta\in(0,\delta_0)$, then there exists $C_{R,\delta}>0$ such that
\[
	\cN_{-\frac 14-E}(B_S) \leq C_{R,\delta} + \cN_{-\frac 1 4-E}(a_{R,\delta}), \quad E>0.
\]
\end{lem}

\begin{proof}
Consider the domains
$U_{R,\delta,1}:=B(R_*)\setminus\Omega_{R,\delta}$ and $U_{R,\delta,2}:=\RR^3\setminus \overline{B(R_*)\mathop{\cup}\Omega_{R,\delta}}$,
where $B(R_*)$ is the ball centered at the origin of radius $R_*\ge R$ chosen sufficiently large in such a way that the two sets have a Lipschitz boundary
and that $S\mathop{\cap} U_{R,\delta,2}=\emptyset$. Introduce the quadratic forms
\begin{align*}
			a_{R,\delta,1}(u)&:= \iiint_{U_{R,\delta,1}}|\nabla u|^2\dd x - \iint_{S\mathop{\cap} U_{R,\delta,1}}|u|^2\dd \sigma,& \dom a_{R,\delta,1} = H^1(U_{R,\delta,1}),\\
			a_{R,\delta,2}(u) & :=  \iiint_{U_{R,\delta,2}} |\nabla u|^2\dd x, &\dom a_{R,\delta,2} = H^1(U_{R,\delta,2}).
\end{align*}
Due to the form inequality $b_S\ge a_{R,\delta}\oplus a_{R,\delta,1}\oplus a_{R,\delta,2}$ one has 
\begin{equation*}
	\cN_{-\frac 1 4-E}(b_S) \leq \cN_{-\frac 1 4-E}(a_{R,\delta}) + \cN_{-\frac 1 4-E}(a_{R,\delta,1}) + \cN_{-\frac 1 4-E}(a_{R,\delta,2}),
	\quad E>0.
\end{equation*}
As $a_{R,\delta,2}$ is non-negative, one has $\mathcal{N}_{-\frac 14-E}(a_{R,\delta,2}) = 0$ for $E>0$.
As $U_{R,\delta,1}$ is bounded Lipschitz, the domain of $H^1(U_{R,\delta,1})$ is compactly embedded
in $L^2(U_{R,\delta,1})$ and 
\[
	\cN_{-\frac 14-E}(a_{R,\delta,1}) \leq \cN_{-\frac14}(a_{R,\delta,1}) =: C_{R,\delta}<\infty, \quad E>0. \qedhere
\]
\end{proof}

Introducing the unitary transfrom
\[
U:L^2(\Omega_{R,\delta})\to L^2\big(\cV_{R,\delta},(r+t\kappa)\,\dd r\,\dd s\,\dd t\big),
\quad
U u:=u\circ \Lambda,
\]
and proceeding literally as in Lemma~\ref{prop:fatub_lay} one shows that the quadratic form $b_{R,\delta}:=a_{R,\delta}\circ (U^{-1})$
in $L^2\big(\cV_{R,\delta},(r+t\kappa)\dd r\,\dd s\,\dd t\big)$ is
	\begin{multline*}
		b_{R,\delta}(u)= \iiint_{\cV_{R,\delta}}
		\bigg((r + t\kappa)\big(|\partial_r {u}|^2 + |\partial_t {u}|^2\big) + \frac{1}{r+t\kappa}|\partial_s {u}|^2\bigg)\dd r \, \dd t\,\dd s\\
		- \int_{\TTT}\int_{R}^{+\infty}|{u}(r,s,0)|^2r\,\dd r\,\dd s
	\end{multline*}
with $\dom b_{R,\delta}=\big\{ u\in L^2\big(\cV_{R,\delta},(r+t\kappa)\dd r\,\dd s\,\dd t\big):
\partial_ru,\,r^{-1}\partial_s u,\,\partial_t u\in L^2\big(\cV_{R,\delta},(r+t\kappa)\dd r\,\dd s\,\dd t\big)\big\}$.
Further, using the unitary transform
\[
V: L^2\big(\cV_{R,\delta},(r+t\kappa)\,\dd r\,\dd s\,\dd t\big)\to L^2(\cV_{R,\delta}),
\quad
(Vu)(r,s,t)=\sqrt{r+t\kappa(s)}\,u(r,s,t),
\quad
\]
and a straightforward computation, almost identical
to the one in Lemma~\ref{prop:unit_2fmlay}, one shows that $b_{R,\delta}$ is unitarily equivalent to
the following quadratic form $c_{R,\delta}$,
		\begin{align*}
			c_{R,\delta}(v) 	&=  \iiint_{\cV_{R,\delta}}\bigg(|\partial_r v|^2 + |\partial_t v|^2 + \frac{1}{(r+t\kappa)^2}\Big(|\partial_s v|^2 - \frac{\kappa^2 +1}4|v|^2\Big) \bigg)\dd r\,\dd s\,\dd t\\
 &\quad+ \iiint_{\cV_{R,\delta}}\Big(\frac{t\kappa''}{2(r+t\kappa)^3} - \frac54\frac{(t\kappa')^2}{(r+t\kappa)^4}\Big)|v|^2\dd r\,\dd s\,\dd t
-
\int_{\TTT}\int_{R}^{+\infty}|v(r,s,0)|^2\dd r\,\dd s\\
&\quad  + \iint_{\TTT\times\RR} \frac{|v(r_{R,\delta}(t),s,t)|^2}{2\big(r_{R,\delta}(t)+t\kappa\big)}\dd s\,\dd t
+
\int_{\TTT}\int_{R}^{+\infty}
\frac{\kappa}{2}\Big(\frac{|v(r,s,-\delta r)|^2}{r-\delta r\kappa} - \frac{|v(r,s,\delta r)|^2}{r+2\delta r\kappa}\Big)\dd r \,\dd s,\\
\dom c_{R,\delta} & =\big\{ v\in L^2(\cV_{R,\delta}):
\partial_rv,\,r^{-1}\partial_s v,\,\partial_t v\in L^2(\cV_{R,\delta})\big\}.
		\end{align*}
In what follows we choose $R>R_0$ and $\delta\in(0,\frac 12\delta_0)$, then in view of Lemma~\ref{lem14}
and of the unitary equivalence we have
	\begin{equation}
	   \label{bs1}
	  \cN_{-\frac14-E}(B_S)\le C_{R,2\delta}+ \cN_{-\frac 1 4-E}(c_{R,2\delta}), \quad E>0.
	\end{equation}
In order to continue we need a suitable lower bound for $c_{R,2\delta}$. First we remark that proceeding in the same spirit as in Subsection \ref{subsec:first-tools-lb}, one can find a constant $A>0$ such that
for all $(R,\delta)$ there holds
\begin{multline*}
	c_{R,2\delta}(v) \geq f_{R,\delta}(v) := \iiint_{\cV_{R,2\delta}}
	\bigg(|\partial_r v|^2 + |\partial_t v|^2 + \frac{|\partial_s v|^2  - \dfrac{\kappa^2 +1 + A\delta}4|v|^2}{r^2(1 + 2\delta\kappa_\infty )^2} \bigg)\dd r\,\dd s\,\dd t\\
	- \int_{\TTT}\int_R^{+\infty}|v(r,s,0)|^2\dd r\,\dd s  - \frac{A}{R}\int_{\TTT}
	\int_{R}^{+\infty} \Big(|v(r,s,-2\delta r)|^2 +|v(r,s,2\delta r)|^2\Big) \dd r\,\dd s
\end{multline*}
with $\dom f_{R,\delta}=\dom c_{R,2\delta}$. We have
\begin{multline*}
	\iiint_{\cV_{R,2\delta}}|\partial_t v|^2\dd r\,\dd s\,\dd t - \frac{A}{R}\int_{\TTT}\int_{R}^{+\infty} \Big(|v(r,s,-2\delta r)|^2 +|v(r,s,2\delta r)|^2\Big) \dd r\,\dd s \\
	\begin{aligned}
	&= \int_{\TTT}\int_{R}^{+\infty}\Big(\int_{-2\delta r}^{2\delta r}|\partial_t v(r,s,t)|^2\dd t - \frac{A}{R}\big(|v(r,s,-2\delta r)|^2 +|v(r,s,2\delta r)|^2\big)\Big)\dd r\dd s\\
	&=\iiint_{\cV_{R,\delta}}|\partial_t v|^2 \dd r\,\dd s\,\dd t\\
	&\quad + \int_{\TTT}\int_{R}^{+\infty}\Big(\int_{-2\delta r}^{-\delta r}|\partial_t v(r,s,t)|^2\dd t - \frac{A}{R}|v(r,s,-2\delta r)|^2 \Big)\dd r\,\dd s \\
	&\quad + \int_{\TTT}\int_{R}^{+\infty}\Big(\int_{\delta r}^{2\delta r}|\partial_t v(r,s,t)|^2\dd t - \frac{A}{R}|v(r,s,2\delta r)|^2 \Big)\dd r\,\dd s \\
	&=:I_1+I_2+I_3.
\end{aligned}
\end{multline*}
Using Eq.~\eqref{eqn:Sob_ineq} with  $a=\min(\delta,A^{-1})R$ and $b=\delta R$ we obtain
\begin{multline*}
	I_2 \geq \Big(1-a\frac{A}{R}\Big)\int_R^{+\infty}\int_{\TTT}\int_{-2\delta r}^{-\delta r}|\partial_t v(r,s,t)|^2\dd t\,\dd s \,\dd r-\frac{2A}{aR} \int_{R}^{+\infty}\int_{\TTT}\int_{-2\delta r}^{-\delta r}|v|^2\dd t\,\dd s\,\dd r\\
	\geq -\frac{2A}{aR}\int_{R}^{+\infty}\int_{\TTT}\int_{-2\delta r}^{-\delta r}|v|^2\dd t\,\dd s\,\dd r.
\end{multline*}
The same reasonning yields
\[
I_3 \geq -\frac{2A}{aR}\int_{R}^{+\infty}\int_{\TTT}\int_{\delta r}^{2\delta r}|v|^2\dd t\,\dd s\,\dd r
\label{eqn:est_I2},
\]
and by choosing $R$ sufficiently large we obtain
\begin{equation}
   \label{eq-i23}
I_2+I_3\geq -\frac{2A}{aR} \iiint_{\cV_{R,2\delta}\setminus\cV_{R,\delta}}|v|^2\dd r\,\dd s\,\dd t
\ge
-\frac{1}{3} \iiint_{\cV_{R,2\delta}\setminus\cV_{R,\delta}}|v|^2\dd r\,\dd s\,\dd t,
\quad v\in \dom f_{R,\delta}.
\end{equation}
Introduce the quadratic forms
\begin{align*}
g_{R,\delta}(v) &:= \iiint_{\cV_{R,\delta}}\bigg(
|\partial_r v|^2 + |\partial_t v|^2 + \frac{1}{r^2(1 + 2\delta\kappa_\infty )^2}\Big(|\partial_s v|^2 - \frac{\kappa^2 +1 + A\delta}4|v|^2\Big)
\bigg) \dd r\,\dd s\,\dd t\\
&\quad - \int_{\TTT}\int_{R}^{+\infty}|v(r,s,0)|^2\dd r\,\dd s,\\
\dom  g_{R,\delta} &=\Big\{ v\in L^2(\cV_{R,\delta}):
\partial_rv,\,r^{-1}\partial_s v,\,\partial_t v\in L^2(\cV_{R,\delta})\Big\},\\
g'_{R,\delta}(v) &:=  \iiint_{\cV_{R,2\delta}\setminus \cV_{R,\delta}}
\Big(|\partial_r v|^2 + |\partial_t v|^2 + \frac{1}{r^2(1 + 2\delta\kappa_\infty )^2}\Big(|\partial_s v|^2 - \frac{\kappa^2 +1 + A\delta}4|v|^2\Big) \bigg) \dd r\,\dd s\,\dd t\\
&\quad - \frac{A}{R}\int_{\TTT}\int_{R}^{\infty}\Big(|v(r,s,-2\delta r)|^2+|v(r,s,2\delta r)|^2\Big)\dd r\,\dd s,\\
\dom  g'_{R,\delta} &=\Big\{ v\in L^2(\cV_{R,2\delta}\setminus \cV_{R,\delta}):
\partial_rv,\,r^{-1}\partial_s v,\,\partial_t v\in L^2(\cV_{R,2\delta}\setminus \cV_{R,\delta})\Big\}.
\end{align*}
Due to the form inequality $f_{R,2\delta}\ge g_{R,\delta}\oplus g'_{R,\delta}$ one has
\begin{equation}
  \label{est00}
	 \cN_{-\frac 1 4 - E}(f_{R,\delta}) \leq \cN_{-\frac 1 4 - E}(g_{R,\delta}) + \cN_{-\frac 1 4 - E}(g'_{R,\delta}), \quad E>0,
\end{equation}
and \eqref{eq-i23} gives
\[
	g'_{R,\delta}(v) \geq -\Big(\frac{\kappa_\infty^2 + 1 + A\delta_0}{4R^2} + \frac13\Big)\|v\|_{L^2(\cV_{R,2\delta}\setminus \cV_{R,\delta})}^2,
\]
By increasing the value of $R$ one obtains $\cN_{-\frac 1 4 - E}(g'_{R,\delta})=0$ for $E>0$.
Using \eqref{bs1} we conclude that for any $\delta \in (0,\frac{\delta_0}{2})$ there exists $R>R_0$
and a constant $B_{R,\delta}>0$ such that
\begin{equation}
    \label{bs2}
\cN_{-\frac14-E}(B_S)\le B_{R,\delta}+ \cN_{-\frac 1 4-E}(g_{R,\delta}), \quad E>0.
\end{equation}
Therefore, it is sufficient to study the eigenvalue counting function for $g_{R,\delta}$.

\subsection{Reduction to two dimensional operators}\label{subsec:fam-2D-op-up}
Use the representation $L^2(\cV_{R,\delta}) \simeq L^2(\cP_{R,\delta}) \otimes L^2(\TTT)$,
then the operator $G_{R,\delta}$ associated with $g_{R,\delta}$
commutes with $1\otimes\cK_S$. As $\cK_S$ has discrete spectrum, it follows that
$G_{R,\delta} \simeq \bigoplus_{n\in\NN}G_{R,\delta}^{[n]}$,
where $G_{R,\delta}^{[n]}$ are the self-adjoint operators in $L^2(\cP_{R,\delta})$ associated with the quadratic forms
\[
	g_{R,\delta}^{[n]}(v) := \int_{\cP_{R,\delta}} \bigg(
	|\partial_r v|^2 + |\partial_t v|^2 + \frac{\lambda_n(\cK_S) - \dfrac{1 + A\delta}4}{r^2(1 + 2\delta\kappa_\infty )^2}\,|v|^2
	\bigg)\dd r\,\dd t- \int_{R}^{+\infty}|v(r,0)|^2\dd r
\]
defined on $\dom g_{R,\delta}^{[n]}=H^1(\cP_{R,\delta})$ and
\[
	\cN_{-\frac14-E}(g_{R,\delta}) = \sum_{n\in\NN} \cN_{-\frac 1 4-E}(G_{R,\delta}^{[n]}), \quad E>0.
\]
By Proposition~\ref{prop:mod1D}, one can increase again the value of $R$
to have, with some $C_0>0$,
\[
	g_{R,\delta}^{[n]}(v) \geq \int_{R}^\infty \int_{-\delta r}^{\delta r} \Big(|\partial_r v|^2 + \bigg(\frac{\lambda_{n}(\cK_S) - \dfrac{1 + A \delta}{4}}{r^2(1+2\delta\kappa_\infty)} -\frac14 - C_0^{-1}e^{-C_0\delta r}\Big)|v|^2\bigg)\dd t\,\dd r.
\]
Set
\[
N_\delta := \max \bigg\{n\in\NN:
\lambda_{n}(\cK_S) - \frac{1 + A \delta}{4}\leq 0 
\bigg\},
\]
then by increasing the value of $R$ once more we arrive at
\[
\frac{\lambda_{n}(\cK_S) - \dfrac{1 + A \delta}{4}}{r^2(1+2\delta\kappa_\infty)}  - C_0^{-1}e^{-C_0\delta r}\ge 0,
\quad r> R, \quad n \geq N_\delta+1.
\]
It follows that for all $n\geq N_\delta+1$ one has $G_{R,\delta}^{[n]} \geq -\frac14$,
and then
\begin{equation}
   \label{count3a}
	\cN_{-\frac14-E}(g_{R,\delta}) = \sum_{n=1}^{N_\delta} \cN_{-\frac 1 4-E}(G_{R,\delta}^{[n]}), \quad E>0.
\end{equation}
To study the case $n\leq N_\delta$ we introduce a parameter $L>1$, denote by $m$ the integer part of $\sqrt{L}$,
and set
\begin{gather*}
r_p := R + \dfrac{pL}{m},
\quad
t_p := \delta r_p,
\quad
p \in\{0, \dots,m\},
\quad
r_{m+1} :=+\infty,\\
\Omega_p := \big\{(r,t)\in\RR^2 : r\in(r_p,r_{p+1}),\  t \in (-t_p,t_p)\big\}\subset \cP_{R,\delta},\quad p\in\{0,\dots,m\},\\
\Omega_{m+1} := \cP_{R,\delta} \setminus \overline{\bigcup\nolimits_{p=0}^{m}\Omega_p}.
\end{gather*}
Introduce the following quadratic forms:
\begin{align*}
			h_{p,\delta}^{[n]}(v)&:= \iint_{\Omega_p}\bigg(|\partial_r v|^2 + |\partial_t v|^2 + \frac{\lambda_n(\cK_S) - \dfrac{1 + A\delta}4}{r^2(1 + 2\delta\kappa_\infty )^2}\,|v|^2 \bigg)\dd r\,\dd t
			-\int_{r_p}^{r_{p+1}}|v(r,0)|^2\dd r,\\
			& \quad \dom h_{p,\delta}^{[n]}(v)=H^1(\Omega_p), \quad p\in\{0,\dots,m\},\\
			h_{m+1,\delta}^{[n]}(v)& := \iint_{\Omega_{m+1}}\bigg(|\partial_r v|^2 + |\partial_t v|^2 + \frac{\lambda_n(\cK_S) - \dfrac{1 + A\delta}4}{r^2(1 + 2\delta\kappa_\infty )^2}\,|v|^2 \bigg)\dd r\,\dd t,\\
			&\quad \dom h_{m+1,\delta}^{[n]}=H^1(\Omega_{m+1}),
\end{align*}
then one has the form inequality $g_{R,\delta}^{[n]}\ge \bigoplus_{p=0}^{m+1} h_{p,\delta}^{[n]}$ implying
\begin{equation}
   \label{count01}
		\cN_{-\frac 1 4-E}(g_{R,\delta}^{[n]}) \leq \sum\nolimits_{p=0}^{m+1} \cN_{-\frac 1 4-E}(h_{p,\delta}^{[n]}).
\end{equation}
We remark first that
\[
	h_{m+1,\delta}^{[n]}(v) \geq \frac{\lambda_1(\cK_S) - \dfrac{1+A\delta}{4}}{R^2(1+2\delta\kappa_\infty)^2}\|v\|_{L^2(\Omega_{m+1})}^2,
\]
hence, we can increase the value of $R$ to get
\[
h_{m+1,\delta}^{[n]}(v) \geq -\tfrac14\|v\|_{L^2(\Omega_{m+1})}^2 \text{ for all } n\in\{1,\dots,N_\delta\},
\]
thus giving
\begin{equation}
   \label{count02}
\cN_{-\frac 1 4-E}(h_{m+1,\delta}^{[n]}) = 0 \text{ for } n\in\{1,\dots,N_\delta\} \text{ and } E>0.
\end{equation}

Now assume that $p\in\{0,\dots,m-1\}$. There holds
\begin{multline*}
h_{p,\delta}^{[n]}(v) \geq \iint_{\Omega_p}\Big(|\partial_r v|^2 + |\partial_t v|^2 \Big)\dd r\, \dd t
-
\int_{r_p}^{r_{p+1}} |v(r,0)|^2\dd r - \epsilon_{p,\delta}\|v\|_{L^2(\Omega_p)}^2\\
=  a_{p,\delta}(v) - \epsilon_{p,\delta} \|v\|_{L^2(\Omega_p)}^2,
\end{multline*}
where $a_{p,\delta}$ is the quadratic form of the operator $N_p \otimes 1 +1\otimes Q_{t_p,N}$
with $N_p$ the Neumann Laplacian in $L^2(r_p,r_{p+1})$, the operator $Q_{t_p,N}$ acting in $L^2(-t_p,t_p)$
and defined in Subsection~\ref{sec:PI} and
\[
\epsilon_{p,\delta} := \Big|\frac{\lambda_1(\cK_S) - \tfrac14(1 + A\delta)}{r_p^2(1+\delta\kappa_\infty)^2}\,\Big|.
\]
Thus, for $E>0$ one has
\[
	\cN_{-\frac 14-E}\big(h_{p,\delta}^{[n]}\big) \leq \cN_{-\frac 14}\big(h_{p,\delta}^{[n]}\big)
	\leq 
	\#\Big\{(l,j)\in\NN_0\times\NN: \dfrac{m^2\pi^2l^2}{L^2} \leq -\dfrac{1}{4} + \epsilon_{p,\delta} - \lambda_j(Q_{t_p,N})\Big\}.
\]
We increase the value of $R$ sufficiently to have $\epsilon_{0,\delta} < \frac 1 4$,
then one has $\epsilon_{p,\delta} < \frac 1 4$ for every $p\in\{0,\dots,m-1\}$.
Furthermore, by Proposition \ref{prop:mod1D} we may additionally assume that $R$ is chosen sufficiently
large to have the estimate $\lambda_j(Q_{t_p,N})\ge 0$ for $j\geq 2$ and the inequalities~\eqref{1d1}.
Then, with the new value of $R$ one has
\begin{multline*}
	\#\Big\{(l,j)\in\NN_0\times\NN: \dfrac{m^2\pi^2l^2}{L^2} \leq -\dfrac{1}{4} + \epsilon_{p,\delta} - \lambda_j(Q_{t_p,N})\Big\}\\
	=
	\#\Big\{l\in\NN_0: \dfrac{m^2\pi^2l^2}{L^2} \leq -\dfrac{1}{4} + \epsilon_{p,\delta} - \lambda_1(Q_{t_p,N})\Big\},
\end{multline*}
and 
\begin{multline*}
\cN_{-\frac 14}\big(h_{p,\delta}^{[n]}\big)
	\leq \#\Big\{l\in\NN_0: \dfrac{m^2\pi^2l^2}{L^2} \leq -\dfrac{1}{4} + \epsilon_{p,\delta} - \lambda_1(Q_{t_p,N})\Big\}\\
\leq \#\Big\{l\in\NN_0 : \dfrac{m^2\pi^2l^2}{L^2} \leq \epsilon_{p,\delta} + C_0^{-1} e^{-C_0 t_p}\Big\}
 \leq 1+ \frac{L}{\pi m}\sqrt{\epsilon_{p,\delta} + C_0^{-1} e^{-C_0 t_p}}
 \leq 1+ c'_{R,\delta}\frac{\sqrt{L}}{r_p}
\end{multline*}
with some $c'_{R,\delta}>0$ independent of $L$ and $n$. Summing over all $p\in\{0,\dots,m-1\}$ we get
\begin{multline*}
	\sum_{p=0}^{m-1}\cN_{-\frac 14}\big(h_{p,\delta}^{[n]}\big) \leq m + c'_{R,\delta}\frac{\sqrt{L}}{R} + c'_{R,\delta}\sqrt{L}\sum_{p=1}^{m-1}\frac{1}{R + L\dfrac{p}m}\\
	\leq m + c'_{R,\delta}\frac{\sqrt{L}}{R} + c'_{R,\delta}m\sqrt{L}\int_{0}^{1}\frac{\dd x}{R + Lx}\\
 = m + c'_{R,\delta}\frac{\sqrt{L}}{R} + c'_{R,\delta}\frac{m}{R\sqrt{L}}\ln\big( 1 + \frac{L}{R}\big)
\leq c''_{R,\delta} \sqrt{L},
\end{multline*}
where $c''_{R,\delta}>0$ is independent of $L$ and $n$. 
Thus, it follows from \eqref{count01} that
\begin{equation}
  \label{count4}
\cN_{-\frac 1 4-E}(g_{R,\delta}^{[n]}) \leq \cN_{-\frac 1 4-E}(h_{m,\delta}^{[n]})+ c''_{R,\delta} \sqrt{L}, \quad E>0.
\end{equation}

\subsection{Reduction to one-dimensional operators}
It remains to find a suitable upper bound for the eigenvalue counting function of $h_{m,\delta}^{[n]}$.
The associated operator $H_{m,\delta}^{[n]}$ can be represented as
\[
H_{m,\delta}^{[n]} = W_{R,L,\delta}^{[n]} \otimes 1 + 1\otimes Q_{t_m,N},
\]
where $Q_{t_m,N}$ acts in $L^2(-t_m,t_m)$ as defined in Subsection~\ref{sec:PI}
and $W^{[n]}_{R,L,\delta}$ is the one-dimensional operator in $L^2(R+L,+\infty)$
associated with the quadratic form
\[
 w_{R,L,\delta}^{[n]}(v):= \int_{R+L}^{+\infty} \bigg(|v'|^2 + \frac{\lambda_n(\cK_S)  - \dfrac{1 + A\delta}{4}}{r^2(1+2\delta\kappa_\infty)^2}|v|^2 \bigg)\dd r,\quad
\dom w_{R,L,\delta}^{[n]} = H^1(R+L,+\infty),
\]
and we get
\[
	\cN_{-\frac 1 4-E}(H_{m,\delta}^{[n]}) = \# \Big\{
	(l,j)\in\NN\times\NN: \lambda_l(Q_{t_m,N}) + \lambda_j(W^{[n]}_{R,L,\delta})\leq -\tfrac14 - E
	\Big\}.
\]
Due to the estimate
\[
	w_{R,L,\delta}^{[n]}(v) \geq -\bigg|\frac{\lambda_1(\cK_S) - \frac{1+A\delta}{4}}{R^2(1+2\delta\kappa_\infty)^2}\bigg| \|v\|_{L^2(R+L,+\infty)}^2
\]
and Proposition~\ref{prop:mod1D} one may increase the value of $R$ to obtain $W_{R,L,\delta}^{[n]} \geq -\frac14$
for all $n$ as well as $\lambda_j(Q_{t_m,N})\ge 0$ for  $j\ge 2$. It follows that
\[
	\Big\{l\in\NN: \lambda_l(W_{R,L,\delta}^{[n]}) \leq -\frac 1 4 - E -  \lambda_j(Q_{t_m,N})\Big\} = \emptyset \text{ for } j\ge2, \quad E>0,
\]
which yields
\[
	\cN_{-\frac 14-E}(H_{m,\delta}^{[n]}) = \cN_{-\frac 1 4-E - \lambda_1(Q_{t_m,N})}(W_{R,L,\delta}^{[n]}), \quad E>0.
\]
With the help of the change of variable $\rho = (R + L)^{-1}r$,
one sees that the quadratic form $w_{R,L,\delta}^{[n]}$
is unitarily equivalent to the form $(R + L)^{-2}z_{\delta}^{[n]}$ in $L^2(1,+\infty)$, where
\[
	z_{\delta}^{[n]}(v) := \int_{1}^{+\infty}
	\bigg(
	|v'|^2 + \frac{\lambda_n(\cK_S) -\dfrac{1 + A\delta}{4}}{(1+2\delta\kappa_\infty)^2\rho^2}|v|^2
	\bigg)
	\dd \rho,\quad
	\dom z_{\delta}^{[n]} = H^1(1,+\infty).
\]
Now we set $L = L(E) := K|\ln E|$ for some $K>0$ to be chosen later on, then
for the respective value $m=m(E)$  we have $\cN_{-\frac 14-E}(H_{m,\delta}^{[n]}) = \cN_{-\mu(E)}(z_{\delta}^{[n]})$, $E>0$,
with
\[
	\mu(E) := \big(R + K|\ln E|\big)^2\big(\tfrac14 + E + \lambda_1(Q_{\delta (R+K|\ln E|),N})\big),
\]
and thanks to Proposition \ref{prop:mod1D} one can estimate
\[
	\Big|\lambda_1(Q_{\delta (R+K|\ln E|),N}) + \tfrac14\Big| \leq C_0^{-1} e^{-C_0 \delta(R + K |\ln E|)} = C_0^{-1} e^{-C_0\delta R} E^{C_0\delta K}.
\]
Hence, by choosing a sufficiently large value of $K$ we may assume that
\[
	\mu(E) = K^2E |\ln E|^2 + o(E|\ln E|^2) \text{ as } E\to 0^+
\]
and then use Proposition~\ref{th:KS88} to describe the asymptotics of $\cN_{-\mu(E)}(z_{\delta}^{[n]})$ as $E\to0^+$.
The substitution into \eqref{count4} and then into \eqref{count3a} gives
\begin{multline*}
	\limsup_{E\rightarrow 0^+}\frac{\cN_{-\frac 1 4-E}(g_{R,\delta})}{|\ln E|} =
	\sum_{n=1}^{N_\delta} \limsup_{E\rightarrow0^+}  \frac{\cN_{-\frac14-E}(g_{R,\delta}^{[n]})}{|\ln E|}\\
	\leq \sum_{n=1}^{N_\delta} \limsup_{E\rightarrow0^+}  \frac{\cN_{-\frac14-E}(h_{R,\delta}^{[n]})}{|\ln E|}
	+ N_\delta c''_{R,\delta} \limsup_{E\rightarrow0^+}  \frac{\sqrt{K|\ln E|}}{|\ln E|}	= \sum_{n=1}^{N_\delta} \limsup_{E\rightarrow0^+}  \frac{\cN_{-\frac14-E}(H_{R,\delta}^{[n]})}{|\ln E|}\\
=\sum_{n=1}^{N_\delta} \limsup_{E\rightarrow0^+}  \frac{\cN_{-\mu(E)}(z_{\delta}^{[n]})}{|\ln E|}
\le \sum_{n=1}^{N_\delta} \limsup_{E\rightarrow0^+}  \frac{\cN_{-\mu(E)}(z_{\delta}^{[n]})}{\big|\ln\mu(E)\big|} \cdot \limsup_{E\rightarrow0^+} \dfrac{\big|\ln\mu(E)\big|}{|\ln E|}\\
= \sum_{n=1}^{N_\delta} \limsup_{E\rightarrow0^+}  \frac{\cN_{-\mu(E)}(z_{\delta}^{[n]})}{\big|\ln\mu(E)\big|}
=\frac{1}{2\pi(1 + 2\delta\kappa_\infty)}\sum_{n = 1}^{N_\delta}
\sqrt{\Big(\delta\big(\tfrac{A}{4}-\kappa_\infty\delta \kappa_\infty^2\big) -\lambda_n(\cK_S)\Big)_+}.
\end{multline*}
In view of \eqref{bs2} we get
\[
	\limsup_{E\rightarrow0^+}\frac{\cN_{-\frac 1 4- E}(B_S)}{|\ln E|}
	\leq \frac{1}{2\pi(1 + 2\delta\kappa_\infty)}\sum_{n = 1}^{N_\delta}
\sqrt{\Big(\delta\big(\tfrac{A}{4}-\kappa_\infty-\delta \kappa_\infty^2\big) -\lambda_n(\cK_S)\Big)_+}.
\]
As the inequality is true for any $\delta\in(0,\frac12\delta_0)$ 
and the right-hand side converges to $k_S$ as $\delta\to0^+$,
we arrive at the sought upper-bound \eqref{prop:born_sup_delta}.

\subsection{Essential spectrum}

In order to complete the proof of Theorem~\ref{thm2} we need to show Eq.~\eqref{asy2ess} for the essential spectrum. Equality \eqref{asy2} shows that $\inf\sigma(B_S)=-\frac 14$,
and it is sufficient to show that $[-\frac 14,+\infty)\subset\sigma(B_S)$,
which can be done in the same way as the respective construction
for $A_S$ in Subsection~\ref{ess1}. Namely, one easily checks
that for a function $\varphi\in \dom B_S$ vanishing in $\RR^3\setminus \Lambda(\cV_{R,\delta})$
one has $B_S\varphi=0$ in $\RR^3\setminus \Lambda(\cV_{R,\delta})$ and
\begin{multline}
   \label{vub}
VU B_S\varphi= \bigg[-\dfrac{\partial^2}{\partial r^2}-\dfrac{\partial}{\partial s} \Big( \dfrac{1}{(r+t\kappa)^2}\dfrac{\partial }{\partial s}\Big)
-\dfrac{\partial^2}{\partial t^2}\\
+\Big(\frac{t\kappa''}{2(r+t\kappa)^3} -\frac54\frac{(t\kappa')^2}{(r+t\kappa)^4} -\dfrac{\kappa^2+1}{4(r+t\kappa)^2}\Big)\bigg] VU\varphi
\end{multline}
in $\big\{ (r,s,t)\in \cV_{R,\delta}: t\ne 0\big\}$.
Pick a $C^\infty$ function $\chi:\RR\to \RR$ with $\chi=0$ on $(-\infty, 0)$ and $\chi=1$ on $(1,+\infty)$
and let $k\ge 0$. For $N>R$, define $\varphi_N\in \dom B_S$ through
$\varphi_N=0$ in $\RR^3\setminus \Lambda(\cV_{R,\delta})$ and
\[
(V U \varphi_N)(r,s,t)=e^{\rmi k r} \exp \big(-\tfrac12 |t|\big) \chi(r-N)\chi(2N-r)\chi(t+N\delta)\chi(N\delta-t),
\quad (r,s,t)\in\cV_{R,\delta},
\]
then a short computation with the help of \eqref{vub} shows that
\[
\lim_{N\to+\infty} \dfrac{\big\|\big(B_S-(k^2-\frac 14)\big)\varphi_N\big\|_{L^2(\RR^3)}}{\|\varphi_N\|_{L^2(\RR^3)}}
=\lim_{N\to+\infty} \dfrac{\big\|VU B_S\varphi_N-(k^2-\frac 14)VU\varphi_N\big\|_{L^2(\cV_{R,\delta})}}{\|VU\varphi_N\|_{L^2(\cV_{R,\delta})}}=0,
\]
which means $k^2-\tfrac14\in \sigma(A_S)$. As $k\ge 0$ is arbitrary, the result follows.

\section*{Acknowledgments} Thomas Ourmi\`eres-Bonafos
is supported by a public grant as part of the ``Investissement d'avenir'' project, reference ANR-11-LABX-0056-LMH, LabEx LMH.

\end{document}